\documentclass{amsart}
\usepackage{amsfonts, amsmath, amsthm, amssymb}
\usepackage{eucal}
\usepackage{todonotes}
\usepackage{mathrsfs}  
\usepackage[all]{xypic}

\usepackage{hyperref}

\newtheorem{theorem}{Theorem}[section]
\newtheorem{innercustomthm}{Theorem}
\newenvironment{customthm}[1]
  {\renewcommand\theinnercustomthm{#1}\innercustomthm}
  {\endinnercustomthm}
\newtheorem{proposition}[theorem]{Proposition}
\newtheorem{proposition-definition}[theorem]
{Proposition-Definition}
\newtheorem{corollary}[theorem]{Corollary}
\newtheorem{lemma}[theorem]{Lemma}

\theoremstyle{definition}
\newtheorem{definition}[theorem]{Definition}

\newtheorem{example}[theorem]{Example}

\newtheorem{remark}[theorem]{Remark}

\theoremstyle{remark}

\newcommand{\Spec}{\mathrm{Spec}\,}
\newcommand{\Sch}{\mathbf{Sch}}
\newcommand{\Spc}{{\cS\!\operatorname{pc}}}

\newcommand{\bN}{\mathbb{N}}

\newcommand{\bQ}{\mathbb{Q}}

\newcommand{\bZ}{\mathbb{Z}}

\newcommand{\cC}{\mathcal{C}}
\newcommand{\cE}{\mathcal{E}}
\newcommand{\cF}{\mathcal{F}}
\newcommand{\cG}{\mathcal{G}}

\newcommand{\cO}{\mathcal{O}}

\newcommand{\cS}{\mathcal{S}}

\newcommand{\cU}{\mathcal{U}}
\newcommand{\cV}{\mathcal{V}}
\newcommand{\cX}{\mathcal{X}}
\newcommand{\cY}{\mathcal{Y}}
\newcommand{\cZ}{\mathcal{Z}}

\newcommand{\sC}{\mathscr{C}}
\newcommand{\sD}{\mathscr{D}}

\newcommand{\Aff}{\mathbf{Aff}}

\newcommand{\TopC}{\Top_\mathbb{C}}

\newcommand{\Topi}{\mathfrak{Top}_\i}

\newcommand{\Sh}{\operatorname{Sh}}

\newcommand{\Div}{\mathrm{Div}}

\newcommand{\gp}{\mathrm{gp}}
\newcommand{\an}{{an}}

\usepackage{xfrac}

\newcommand{\et}{\acute{e}t}

\def\Top{\mathbf{Top}}
\def\Sch{\mathbf{Sch}}

\renewcommand{\i}{\infty}

\def\Pro{\operatorname{Pro}}

\def\Fun{\operatorname{Fun}}
\DeclareMathOperator{\Psh}{Psh}

\def\Hshi{\mathbb{H}\mathrm{ypSh}_\i}
\def\colim{\underrightarrow{\mathrm{colim}\vspace{0.5pt}}\mspace{4mu}}

\renewcommand{\lim}{\varprojlim\mspace{3mu}}
\def\blank{\mspace{3mu}\cdot\mspace{3mu}}

\def\Profs{\operatorname{Prof}\left(\Spc\right)}
\def\Prof{\operatorname{Prof}}
\def\Shape{\mathit{Shape}}

\def\Lan{\operatorname{Lan}}

\def\Top{\mathbf{Top}}
\def\Set{\mathit{Set}}

\def\longlongrightarrow{-\!\!\!-\!\!\!-\!\!\!-\!\!\!-\!\!\!-\!\!\!\longrightarrow}

\def\rrrarrow{\hspace{.05cm}\mbox{\,\put(0,-3){$\rightarrow$}\put(0,1){$\rightarrow$}\put(0,5){$\rightarrow$}\hspace{.45cm}}}
\renewcommand{\i}{\infty}

\newcommand\radice[2]{\hspace{-1.5pt}\sqrt[\uproot{2}#1]{#2}}

\newcommand{\Map}{\mathbb{M}\mathrm{ap}}
\newcommand{\Aut}{\mathbf{Aut}}

\newcommand{\triv}{\mathit{triv}}
\newcommand{\pro}{\mathit{pro}}

\def\colim{\underrightarrow{\mathrm{colim}\vspace{0.5pt}}\mspace{4mu}}

\title{On the profinite homotopy type of log schemes}

\author{David Carchedi}
\address{Department of Mathematical Sciences\\
George Mason University\\
4400 University Drive, MS:  3F2\\
Exploratory Hall\\
Fairfax, Virginia  22030\\
USA}
\email{davidcarchedi@gmail.com}
\author{Sarah Scherotzke}
\address{Mathematical Institute of the University of M\"unster, 
Einsteinstrasse 62, 
48149 M\"unster,
Germany}
\email{scherotz@math.uni-muenster.de}
\author{Nicol\`o Sibilla}
\address{School of Mathematics, Statistics and Actuarial Sciences\\ 
University of Kent\\ 
Canterbury, Kent CT2 7NF\\
UK}
\email{N.Sibilla@kent.ac.uk}
\author{Mattia Talpo}
\address{Department of Mathematics\\
University of Pisa\\
Largo Bruno Pontecorvo 5\\
56127 Pisa PI\\
Italy}
\email{mattia.talpo@unipi.it}


\begin{document}

\maketitle

\begin{abstract}
We complete the program, initiated in \cite{KNIRS}, to compare the many different possible definitions of the underlying homotopy type of a log scheme. We show that, up to profinite completion, they all yield the same result, and thus arrive at an unambiguous definition of the profinite homotopy type of a log scheme. Specifically, in \cite{KNIRS}, we define this to be the profinite \'etale homotopy type of the infinite root stack, and show that, over $\mathbb{C},$  
this agrees up to profinite completion with the Kato-Nakayama space. Other possible candidates are the profinite shape of the  Kummer \'etale site $X_{k\et},$ or of the representable \'etale site of $\radice{\i}{X}.$  Our main result is that all of these notions agree, and moreover the \emph{profinite} \'etale homotopy type of the infinite root stack is not sensitive to whether or not it is viewed as a pro-system in stacks, or as an actual stack (by taking the limit of the pro-system). We furthermore show that in the log regular setting, all these notions also agree with the \'etale homotopy type of the classical locus $X^{\triv}$ (up to  an appropriate completion). We deduce that, over an arbitrary locally Noetherian base, the \'etale homotopy type of $\mathbb{G}_m^N$ agrees with that of $B\boldsymbol{\mu}_\i^N$ up to completion.  
\end{abstract}

\section{Introduction}
The category of logarithmic (log) schemes is an enlargement of the category of schemes. Initially designed for applications to arithmetic geometry, log  geometry has proved to be an invaluable tool in a broad array of mathematical areas, including algebraic and symplectic geometry, mirror symmetry and homotopy theory. 

In this paper we  complete the program we initiated in \cite{KNIRS}, and aimed at charting  the topology of log schemes.  
We obtain comparison results linking all  different models of the underlying homotopy type of a log scheme. Additionally, we relate this to the \'etale homotopy type of the classical locus. 

\subsection{The \'etale homotopy type}
If $X$ is a classical scheme, there are several ways to extract from it topological information. If $X$ is  a complex scheme of finite type, we can consider  its analytification $X^{\an}=X(\mathbb{C})$. Over a general base,  the closest  approximation to the underlying topological space of $X$ is given by the \emph{\'etale homotopy type}  of Artin--Mazur and Friedlander. From a modern perspective, this is  an instance of the general formalism of the \emph{shape} of a topos, which is applied in this case to the small \emph{\'etale topos} of $X$. When these notions overlap, we have comparison results ensuring information flow across these different perspectives. Over $\mathbb{C}$, the comparison between the \'etale homotopy type and the analytification is a generalization of Riemann's existence theorem. These comparison theorems are both computationally and conceptually powerful: they show that there is, at bottom, only one meaningful notion of the underlying topology of a scheme, 
and that there is a variety of different techniques that we can deploy to study it. 

In this article, we  describe the different models for the underlying homotopy type of a log scheme, and prove comparisons results between them. Log geometry is a powerful formalism, but teasing out the underlying geometric information of log schemes is not easy. One reason is that the definition of a log structure combines geometric and combinatorial data,    
 which might keep  track for instance of the combinatorics of the  compactifying divisors, or of the singularities of 
the central fiber of a degeneration.  
To get around this issue, several more classically geometric objects have been designed to capture the geometric properties  of log schemes: this includes the  \emph{Kato--Nakayama space}, a beautiful construction  available over $\mathbb{C}$, that attaches to a log scheme a topological space (which is homeomorphic to a topological manifold with corners, under log regular assumptions); the \emph{Kummer \'etale topos}, which is the category of sheaves on a log analogue of the small \'etale site; and the \emph{infinite root stack}, due to Talpo and Vistoli \cite{TV}, which is a limit of  tame algebraic stacks.

It turns out that, at a topological level, all these different incarnations of the geometry of a log scheme yield equivalent objects. This is our main theorem. As in the classical case, our comparison result does not have purely a conceptual import: it is also a valuable computational tool. In the next section we give a more detailed account of our main result. 
Then in  section \ref{applications}  we will explain applications of our work, some of which will appear in a future companion article.  

\subsection{The main result}
Recall that if $\mathcal{E}$ is a $\infty$-topos, the \emph{shape} of $\mathcal{E},$ denoted by
$\Shape(\mathcal{E})$, is, in a precise sense, the best approximation of $\mathcal{E}$ by a pro-space. 
We mentioned already that  
  the \'etale homotopy type is a special case of it.   
As we will explain below, the notion of shape is also 
 one of the key tools for defining the underlying homotopy type of log schemes. 

Many of our proofs will depend on delicate descent arguments. Thus it will be essential for us to work with a sufficiently flexible formalism, 
 capable of keeping track of  
all the higher homotopies involved in descent statements.  
Therefore, it will be indispensable for us to work within the framework of $\i$-categories. Moreover, to deal with higher categorical descent statements, we will need the framework of $\i$-topoi. 
We will study sheaves of spaces on various sites: contrary to classical references we will work with \emph{$\infty$-topoi of sheaves of spaces, or $\infty$-groupoids}, rather than just with sheaves of sets.  
We remark that in the case of a scheme, the shape of the small \'etale $\infty$-topos  
is the modern formulation of the \'etale homotopy type.  

In order to give a more precise account of our work, let us start by  introducing in some more detail the cast of characters which will play a role in our comparison result. 
The \emph{Kummer \'etale site} of a log scheme $X$  is the analogue in the log setting of the classical \'etale site. In the divisorial case, Kummer \'etale maps are maps which 
restrict to ordinary \'etale maps away from the divisor, but can be tamely ramified along it.  
This definition is well aligned  with the general logarithmic  philosophy, 
 which roughly gives us a way to  regard  objects that develop singular behavior along the log   locus as being still smooth.   
We denote by $\Sh\left(X_{k\et}\right)$  the Kummer \'etale topos of $X$. Since the shape of the \'etale topos of a classical scheme is precisely its \'etale homotopy type, the shape of $\Sh\left(X_{k\et}\right)$ is a natural candidate for the notion of the underlying homotopy type of a log scheme $X$. 

The \emph{infinite root stack} of $X$ is an inverse limit of the root construction along the support of the log structure.  It was introduced by Talpo--Vistoli in \cite{TV}.  It can be regarded both as a stack, or as a pro-object in stacks, by viewing it as a formal limit of finite root stacks. We are ultimately interested in shapes, 
and the two point of views on the infinite root stacks will yield genuinely different answers in general. One of the results in this paper however shows the difference however is, in a precise sense, mild, as it is erased after profinite completion. There is a third possible way of extracting an underlying homotopy type from a log scheme which
emerges naturally from  work of Talpo and Vistoli, namely the shape of the topos of sheaves on the representable \'etale site of the infinite root stack. 

Finally over $\mathbb{C}$, Kato and Nakayama gave a recipe to attach to a log scheme 
$X$ an actual topological space $X^{log}$, called the Kato--Nakayama (KN) space. In the divisorial case, the KN space is homeomorphic to the bordification of the complement of the divisor; equivalently, it is homeomorphic to the real oriented blow-up of $X$ along the divisor. Over $\mathbb{C}$, the KN space provides another possible definition of the underlying homotopy type of a log scheme.

Summarizing the above discussion,  we obtain five distinct possible definitions of the underlying homotopy type of a log schemes. We list them below: 
 
\begin{enumerate}
 \item  The shape of the $\i$-topos $\Sh\left(X_{k\et}\right).$
 \item  The \'etale homotopy type of the $\i$-root stack $\radice{\i}{X}.$
 \item  The \'etale homotopy type of the $\i$-root stack $\radice{\i}{X}^{pro}$ regarded as a pro-object. 
 \item The shape of the $\i$-topos of \'etale representable maps to $\radice{\i}{X}$, $\Sh\left(\radice{\i}{X}^{rep}_{\et}\right).$
  \item  Over $\mathbb{C}$, the  homotopy type of the Kato-Nakayama space of $X.$
\end{enumerate}

Our main result completes the program initiated in \cite{KNIRS} by comparing  these different homotopy types. We state the general comparison result as the following statement.  

\begin{customthm}{A}[Theorem \ref{thm:prothesame}, Theorem \ref{protruncated}]
\label{theoremA}
Let $X$ be a log scheme. 
\begin{enumerate}
\item 
The homotopy types (1) , (2)  , (3)  and (4) are all equivalent \emph{up to profinite completion}. 
\item Over $\mathbb{C}$, they are also all equivalent to (5) up to profinite completion. \end{enumerate}
\end{customthm}

Theorem \ref{theoremA} combines our results in this article, with previous results obtained by us and our collaborators. Namely, the comparison between (1) and (4)  follows from work of Talpo--Vistoli, and holds before passing to profinite completions; over $\mathbb{C}$ the comparison between (5) and (3) was the main result of our previous work \cite{KNIRS}. 

\begin{remark}
It is worth mentioning that the comparison between (4) and (2) is \emph{NOT} a tautology since $\radice{\i}{X}$ is not Deligne-Mumford--- it is not even algebraic! The proof is in fact quite involved.
\end{remark}

The heavy lifting done in this paper is in establishing the comparison between (2), (3) and (4). This is far from obvious,  and requires grappling with delicate technical issues. The necessary 
arguments   make up the bulk of the present paper. Section \ref{bettidescent} contains some preliminary results on descent for Betti stacks. Namely, we prove that for $V$ any $\pi$-finite space, its \'etale Betti stack $\Delta^{\et}\left(V\right)$ satisfies fpqc-descent. Leveraging this, in Section  
\ref{shapecomparison1} we prove 
that (2) and (3) are equivalent after profinite completion;  the comparison between (3) and (4) is established in section 
\ref{shapecomparison2}. We stress that, before completion, the constructions (2), (3) and (4) yield genuinely different pro-spaces. 
 
 \subsection{The homotopy type of the classical locus} If $X$ is a log scheme, its \emph{classical locus} is the largest open subscheme where the log structure is trivial.  We denote it by  $X^{\triv}$.   
 Log structures appear naturally when working with compactifications.  If $X \supset U$ is a sufficiently well-behaved compactification of an open variety $U$, log geometry allows us to recover information on  $U$ by working relative to the compactifying divisor: $X$ is equipped with a natural log structure which keeps track of the divisor at infinity, and  has the property that  
$X^{\triv}=U$.  
A central organizing principle in the area is that, under suitable assumptions,  log invariants of $X$ should coincide with the corresponding classical invariants of $U$. As an early instance of this circle of ideas, we should mention Grothendieck's theorem stating that the cohomology of   differential forms with log poles along a divisor computes the de Rham cohomology of the complement of the divisor. 

In this paper we focus on what is, in a precise sense, 
the most fundamental topological invariant of a log scheme $X$: namely, its  (profinite)  homotopy type. By Theorem \ref{theoremA} this is a well-defined notion, as up to profinite completion all different constructions converge to yield the same answer. 
To differentiate it from the classical \'etale homotopy type, we will refer to it  
as the (profinite)  \emph{log homotopy type} of $X$. 
We state our second main result below. We assume that 
$X$ is \emph{log regular}, which is a generalization of regularity for log schemes: the result cannot be expected to hold under weaker hypotheses.\footnote{Here and elsewhere we also make   further standard mild assumptions on $X$, we refer the reader to the main text for a complete account of this. We  remark however that we do not work over a field, but over a  general base scheme}. We remark that our theorem builds on recent results contained in Berner's thesis \cite{Berner}. 

\begin{customthm}{B}[Corollary \ref{cor:Berner}]
\label{theoremB}
Let $X$ be a log regular log scheme. 
\begin{enumerate}
\item Let $\ell$ be a prime which is invertible on $X$. Then the $\ell$-profinite completion of  the log homotopy type of $X$ is equivalent to the $\ell$-profinite completion of the \'etale homotopy type of $X^{\triv}$. 
\item In characteristic zero, the profinite completion of the log homotopy type of $X$ is equivalent to the profinite completion of the  \'etale homotopy type of $X^{\triv}$. 
\end{enumerate}
\end{customthm}

Leveraging all of the comparison results, we are able to prove the following theorem  comparing the \'etale homotopy type of $\mathbb{G}_m$ with that of $B\boldsymbol{\mu}_\i.$

\begin{theorem}
Let $S$ be a locally Noetherian scheme, and denote by $\boldsymbol{\mu}_\i$ the affine group scheme over $S$ $$\underset{n} \lim \boldsymbol{\mu}_n.$$ Let $\ell$ be a prime invertible on $S,$ and $N$ any non-negative integer. Then there is an equivalence of $\ell$-profinite spaces
$$\Pi^{\et}_\i\left(\mathbb{G}_m^N\right)^{\wedge}_{\ell} \simeq  \Pi^{\et}_\i\left(B\boldsymbol{\mu}_{\i}^N\right)^{\wedge}_{\ell}.$$ Moreover, in characteristic zero, this holds up to profinite completion.
\end{theorem}
   
\subsection{Applications and future work}
\label{applications}
The \'etale homotopy type  is a very rich invariant of  schemes. Both the \'etale fundamental group, and the \'etale cohomology of locally constant sheaves, can be computed from the \'etale homotopy type. In fact, under suitable assumptions,  all \emph{topological} invariants of a scheme should be encoded, up to completion, in its \'etale homotopy type. 
As a prominent example of this kind of thinking, let us mention the  beautiful work of Friedlander which recasts  \'etale K-theory  in terms of the  topological K-theory  of the \'etale homotopy type \cite{friedlander1980etalek}.


One of the  motivations behind the present project is to transfer this perspective to the logarithmic setting. In order to do so, it is necessary first of all  to gain a finer understanding of the underlying homotopy type of a log scheme. In this article we accomplish this first step via our Theorem \ref{theoremA} and \ref{theoremB}. 
In a companion article, we hope to pursue applications of  our results to the logarithmic  version of \'etale K-theory, which was introduced by Nizio{\l} in \cite{Ni1}. 
\subsection*{Acknowledgments}
We are grateful for discussions with Alex Betts, Bhargav Bhatt, Elden Elmanto, Marc Hoyois, Jacob Lurie, Tomer Schlank, and Bertrand To\"{e}n. We would also like to thank the Max Planck Institute for Mathematics (MPIM), and the math department of George Mason University for helping accommodate this collaboration. D.C. would also like to thank the MPIM for support during his sabbatical. S.S. was supported by the National Science Foundation under Grant No. 1440140, while the author was in residence at the Mathematical Sciences Research Institute (MSRI) in Berkeley, California. N.S and S.S thank the MSRI, Berkeley, for excellent working conditions. M.T. was partly supported by EPSRC grant EP/R013349/1.

\section{Preliminaries}
\subsection{The $\i$-category of pro-spaces and its localizations}\label{subsec:proobj} 
We will work within the framework of $\i$-categories (aka quasicategories) and follow the notational conventions and terminology of \cite{htt}. We will also assume the reader is familiar with the the salient points of the theory of pro-objects in $\i$-categories. For more details on this theory see e.g. \cite[Appendix E]{SAG}, \cite[Section 3]{dagxiii}, \cite{prohomotopy} and \cite[Section 2.1]{reletale}.

Given an $\i$-category $\sC,$ the $\i$-category $\Pro\left(\sC\right)$ of pro-objects is defined by a universal property, making rigorous the idea that objects of $\Pro\left(\sC\right)$ are \emph{formal cofiltered limits} of objects in $\sC.$

\begin{definition} \label{def:pro} The $\infty$-category $\Pro\left(\sC\right)$ has cofiltered limits, is equipped with a fully faithful functor 
\begin{equation} \label{eq:pro-yoneda}
j: \sC \rightarrow \Pro\left(\sC\right)
\end{equation} 
and if $\sD$ is an $\infty$-category admitting small cofiltered limits, then the precomposition with $j$ induces an equivalence of $\infty$-categories:

$$\Fun_{\mathit{cofilt}}\left(\Pro\left(\sC\right), \sD\right) \rightarrow \Fun\left(\sC, \sD\right)$$
where $\Fun_{\mathit{cofilt}}\left(\Pro\left(\sC\right), \sD\right)$ is the full subcategory of $\Fun\left(\Pro\left(\sC\right), \sD\right)$ on those functors that preserve small cofiltered limits.
\end{definition}

We will make continual use of the following two propositions:

\begin{proposition} \label{prop:lex} \cite[Proposition 3.1.6]{dagxiii} Suppose that $\sC$ is accessible and has finite limits, then $\Pro\left(\sC\right)$ is equivalent to the full subcategory of $\Fun\left(\sC, \Spc\right)^{op}$ on those functors which are left exact and accessible.
\end{proposition}

\begin{proposition} \label{prop:ladj}  Let $f: \sD \rightarrow \sC$ be a functor between accessible $\i$-categories with finite limits, and suppose that that $f$ is left exact, then the functor $$\Pro\left(f\right): \Pro\left(\sD\right) \rightarrow \Pro\left(\sC\right)$$ has a left adjoint $$f^*: \Pro\left(\sC\right) \rightarrow \Pro\left(\sD\right),$$ given by restriction along $f.$
\end{proposition}

\begin{definition}
The $\i$-category of \textbf{pro-spaces} is the $\i$-category $\Pro\left(\Spc \right).$ 
\end{definition}

Note that, given any subcategory $\sC \subseteq \Spc$ closed under finite limits and retracts, it follows from 
Proposition \ref{prop:ladj} that the canonical inclusion $$\Pro\left(\sC\right) \hookrightarrow  \Pro\left(\Spc\right)$$ has a left adjoint $\widehat{\left(\blank\right)}_{\sC}.$

\begin{definition}
A space $X$ is called \textbf{$\pi$-finite} if it has finitely many connected components, and finitely many non-trivial homotopy groups, all of which are finite. The $\i$-category $\Pro\left(\Spc^\pi \right)$ is the $\i$-category of \textbf{profinite spaces}, and is denoted by $\Profs.$ We denote the left adjoint $\widehat{\left(\blank\right)}_{\Spc^\pi}$ simply by $\widehat{\left(\blank\right)}$ and call it the \textbf{profinite completion} functor.
\end{definition}

\begin{definition}
Fix a prime $\ell.$ A space $X$ is called \textbf{$\ell$-finite} if it is $\pi$-finite and all its homotopy groups are finite $\ell$-groups. The $\i$-category $\Pro\left(\Spc^{\pi-\ell} \right)$ is the $\i$-category of \textbf{$\ell$-profinite spaces}, and is denoted by $\Prof_{\ell}\left(\Spc\right).$ We denote the left adjoint $\widehat{\left(\blank\right)}_{\Spc^{\pi-\ell}}$  by $\left(\blank\right)^{\wedge}_{\ell}$ and call it the \textbf{$\ell$-profinite completion} functor.
\end{definition}

\begin{proposition}\label{prop:ellsame}
Let $f:\cX \to \cY$ be a map in $\Pro\left(\Spc\right),$ then $f$ induces an equivalences
$$\cX^{\wedge}_{\ell} \to \cY^{\wedge}_{\ell}$$ if and only if for all $n,$
$$\cY\left(K\left(\mathbb{Z}/\ell,n\right)\right) \to \cX\left(K\left(\mathbb{Z}/\ell,n\right)\right)$$ is an equivalence.
\end{proposition}

\begin{proof}
This follows immediately from \cite[Corollary 7.3.7]{prohomotopy}.
\end{proof}

\begin{definition}
Denote by $\Spc_{< \i}$ the full subcategory of $\Spc$ on those spaces which have only finitely many non-trivial homotopy groups. The $\i$-category $\Pro\left(\Spc_{< \i}\right)$ is called the $\i$-category of \textbf{protruncated spaces}. We denote this $\i$-category by $\Pro_{< \i}\left(\Spc\right)$ and the associated localization functor by $\left(\blank\right)^{\wedge}_{< \i}.$
\end{definition}

\subsection{Shape theory of $\i$-topoi}
Recall that an \textbf{$\i$-topos} is an $\i$-category $\cE$ which arises as a left exact localization of an $\i$-category of presheaves of spaces $\Psh\left(\sC\right)$ on a small $\i$-category $\sC.$ In other words there is a fully faithful inclusion $$i:\cE \hookrightarrow \Psh\left(\sC\right)$$ which has a left adjoint $a,$ and this left adjoint moreover is left exact, i.e. preserves finite limits. The prototypical examples arise by equipping $\sC$ with a Grothendieck topology, and letting $\cE$ be the full subcategory on those presheaves that satisfy descent with respect to \v{C}ech-covers, or hyperdescent.

A \textbf{geometric morphism} $f:\cE \to \cF$ between $\i$-topoi is a pair of adjoint functors $$f^* \dashv f_*,$$ such that $f^*$ is left exact. These are the $1$-morphisms in the $\i$-category $\Topi$ of $\i$-topoi.

\begin{example}
The $\i$-topos $\Spc$ of spaces is a terminal object in $\Topi.$ This follows since a colimit preserving functor $f^*:\Spc \to \cE$ with $\cE$ cocomplete is completely determined by where it sends the one-point space, and moreover has a  right adjoint $f_*.$ Furthermore, if $f^*$ is left exact, it must send the one-point space to a terminal object in $\cE,$ since the one-point space is terminal in $\Spc.$ The essentially unique geometric morphism $$\cE \to \Spc$$ is denoted by $\Delta_{\cE} \dashv \Gamma_{\cE}.$ Concretely, $$\Gamma\left(E\right)\simeq \Map_{\cE}\left(1,E\right).$$ If $\cE$ arises as sheaves on a site, $\Delta_{\cE}\left(X\right)$ is the constant sheaf with values $X,$ i.e. the sheafification of the constant presheaf.
\end{example}

Given a space $X \in \Spc,$ there is a canonical equivalence $\Spc/X \simeq \Psh\left(X\right),$ and there is an induced fully faithful functor \cite[Remark 6.3.5.10, Theorem 6.3.5.13, and Proposition 6.3.4.1]{htt}
$$\Spc/\blank:\Spc \hookrightarrow \Topi.$$ By the universal property of $\Pro\left(\Spc\right)$ this extends to a unique cofiltered limit preserving functor
$$\Spc/\blank:\Pro\left(\Spc\right) \to \Topi.$$

\begin{remark}
The above functor is not fully faithful, but its restriction to $\Profs$ is by \cite[Appendix E.2]{SAG}.
\end{remark}

By \cite[Remark 7.1.6.15]{htt}, this functor has a left adjoint $$\Shape:\mathfrak{Top}_\i \to \Pro\left( \Spc \right).$$ In fact, it has a relatively simple description: to describe a pro-space, it suffices to give a left exact accessible functor from $\Spc$ to itself. The functor $\Shape$ sends an $\i$-topos $\cE$ to the functor $$\Spc \stackrel{\Delta_{\cE}}{\longlongrightarrow} \cE \stackrel{\Gamma_{\cE}}{\longlongrightarrow} \Spc.$$

\begin{definition} Let $\cE$ be an $\infty$-topos, then the pro-space $$\Shape\left(\cE\right)$$ called the \textbf{shape} of $\cE.$
\end{definition}

Applying natural completion functors to the shape functor gives natural variants. For example, the profinite completion of $\Shape\left(\cE\right)$ is referred to as the \emph{profinite shape} of $\cE,$ etc.

The shape of an $\i$-topos is, in a precise sense, its best approximation by a pro-space. It serves as a suitable notion of underlying homotopy type of an $\i$-topos. For example, if a \emph{topological} space $T$ has the homotopy type of a CW-complex, it follows from \cite[Remark A.1.4]{higheralgebra} that the shape of its $\i$-topos of sheaves of spaces $\Sh\left(T\right)$ is that of the underlying homotopy type of $T.$ For further geometric intuition about this functor, we refer the reader to \cite[Section 2.2.1]{reletale}.

\subsection{\'Etale homotopy types}
For $X$ any scheme, its \emph{\'etale homotopy type} is the pro-space
$$\Shape\left(\Sh\left(X_{\et}\right)\right)$$
where $\Sh\left(X_{\et}\right)$ is the $\i$-category of sheaves of spaces on the small \'etale site of $X.$

In \cite{etalehomotopy}, the first author extends this definition to arbitrary sheaves of spaces on the large \'etale site. We give a rapid recollection:

Let us take $\Aff'$ to be a suitable \emph{small} subcategory of affine schemes, closed under finite limits and \'etale morphisms, and containing the empty scheme. Then there is a unique colimit preserving functor
$$\Sh\left(\left(\blank\right)_{\et}\right):\Sh\left(\Aff',\et\right) \to \Topi$$ which sends each affine scheme $S$ to sheaves of spaces on its small \'etale site. Concretely, for an arbitrary sheaf of spaces $\cX$ on $\left(\Aff',\et\right),$ the $\i$-topos $\Sh\left(\cX_{\et}\right)$ is the colimit
$$\underset{T \to \cX} \colim \Sh\left(T_{\et}\right),$$ computed in $\Topi$, where the colimit ranges over all maps from affine schemes (in $\Aff'$) into $\cX.$

\begin{definition}
The \'etale homotopy type $\Pi^{\et}_\i\left(\cX\right)$ of $\cX$ is the shape of the $\i$-topos $\Sh\left(\cX_{\et}\right).$
\end{definition}

Since it involves a colimit, the above definition is a bit opaque for stacks which are not schemes (or at least Deligne-Mumford stacks). However, there is a much simpler description given also in \cite{etalehomotopy}, which we now recall.

\begin{definition}
Let $V$ be a space in $\Spc.$ Denote by $\Delta^{\et}\left(V\right)$ its constant stack on $\left(\Aff',\et\right).$ This stack is called the \textbf{Betti stack} associated to $V.$
\end{definition}

\begin{theorem}\cite[Theorem 2.40]{etalehomotopy}
For a sheaf of spaces $\cX$ on $\left(\Aff',\et\right),$ its \'etale homotopy type is given by
\begin{eqnarray*}
 \Pi^{\et}_\i\left(\cX\right):\Spc &\to& \Spc\\
 V &\mapsto& \Map\left(\cX,\Delta^{\et}\left(V\right)\right).
 \end{eqnarray*}
\end{theorem}

\subsection{Higher Deligne-Mumford Stacks}
Recall that for a scheme $X,$ its small \'etale site $X_{\et}$ is the category of \'etale maps $\Spec A \to X,$ with $A$ a commutative ring, equipped with the \'etale topology. There is a canonical sheaf of rings $\cO_X$ on $X_{\et}$ which assigns to such an object the ring $A,$ and the stalks are 
strictly Henselian. This makes $\left(\Sh\left(X_{\et}\right),\cO_X\right)$ into a strictly Henselian ringed $\i$-topos. 

\begin{definition}
Let $\left(\cE,\cO_{\cE}\right)$ be a strictly Henselian ringed $\i$-topos. It is \textbf{Deligne-Mumford} if there exists a set of objects $\left(E_\alpha\right)_\alpha$ in $\cE$ such that $$\underset{\alpha} \coprod E_\alpha \to 1$$ is an epimorphism and such that for all $\alpha,$
$$\left(\cE/E_{\alpha},\cO_{\cE}|_{E_\alpha}\right)$$ is equivalent to $\left(\Sh\left(X_{\et}\right),\cO_X\right)$ for some scheme $X$ (depending on $\alpha$).
\end{definition}

\begin{definition}
A sheaf of spaces $\cX$ on $\left(\Aff',\et\right)$ is \textbf{Deligne-Mumford} (or a higher Deligne-Mumford stack) if it is equivalent to the functor of points of a Deligne-Mumford strictly Henselian ringed $\i$-topos $\left(\cE,\cO_{\cE}\right)$, i.e. if
$$\cX:S \mapsto  \Map\left(\left(\Sh\left(S_{\et}\right),\cO_{S}\right),\left(\cE,\cO_{\cE}\right)\right),$$ where the space of maps is in the $\i$-category of strictly Henselian ringed $\i$-topoi.
\end{definition}

\begin{definition}
Let $\cX$ be Deligne-Mumford. Then its \textbf{small \'etale site} $\cX_{\et}$ is the $\i$-category of (not-necessarily representable) \'etale maps $\Spec A \to \cX,$ equipped with the induced \'etale topology.
\end{definition}

\begin{remark}
By \cite[Lemma 2.25]{etalehomotopy}, we have an identification $$\Sh\left(\cX_{\et}\right) \simeq \cE.$$
\end{remark}

\begin{proposition}
Let $\cX$ be Deligne-Mumford. Then $\i$-category $\mathfrak{DM}^{\et}_{\cX}$ of (not-necessarily representable) \'etale maps over $\cX$ (with domain any other higher Deligne-Mumford stack) is equivalent to 
$\Sh\left(\cX_{\et}\right).$
\end{proposition}

\begin{proof}
This follows immediately from the previous remark and \cite[Remark 6.3.5.10]{htt}
\end{proof}

In particular, for any scheme $T,$ we have that $\Sh\left(T_{\et}\right)$ can be identified with the $\i$-category $\mathfrak{DM}^{\et}_{T}$ of \'etale maps $\cX \to T,$ with $\cX$ a higher Deligne-Mumford stack.

\subsection{Brief review on some notions in log geometry}

We assume that the reader is familiar with the basic notions of logarithmic geometry (see for example \cite[Appendix]{KNIRS}). We include a brief review of some of the objects that play an important role in this paper.

Let $X$ be a fine saturated log scheme, with a Deligne-Falterings log structure $L\colon A\to \Div_X$ (i.e. $A$ is an \'etale sheaf of monoids on $X$, and $L$ is a symmetric monoidal functor with trivial kernel to the category of stack of pairs $(L,s)$ of line bundles with a section), and $n\in \bN$.

\begin{definition}
The $n$-th root stack $\radice{n}{X}$ is the stack on the category of schemes over $X$, that assigns to a scheme $f\colon T\to X$ the groupoid of Deligne-Falterings structures $f^*\frac{1}{n}A\to \Div_X$ extending $f^*L\colon f^*A\to \Div_T$. We denote also by $\radice{n}{X}$ the underlying stack on all schemes. By above, there is a canonical map $$\radice{n}{X} \to X.$$
\end{definition}

If $n\mid m$ there is a natural map $\radice{m}{X}\to \radice{n}{X}$, given by restricting along the inclusion $\frac{1}{n}A\to \frac{1}{m}A$. The assignment $n \mapsto \radice{n}{X}$ assembles into a pro-object in stacks, where $\bN$ is ordered by divisibility.

\begin{definition}
We denote the corresponding pro-object by $\radice{\infty}{X}_{pro}.$ The actual limit in stacks is called the \textbf{infinite root stack} of $X$ and is denoted by $\radice{\infty}{X}.$
\end{definition}

\begin{remark}
In \cite{KNIRS}, we refer to the pro-object as the infinite root stack instead.
\end{remark}

If the log structure of $X$ is coherent (in particular if $X$ is fine and saturated), the stacks $\radice{n}{X}$ are algebraic. More precisely, if $X\to \Spec \bZ[P]$ is a Kato chart, where $P$ is a fine saturated sharp monoid, then the $n$-th root stack  $\radice{n}{X}$ is isomorphic to the quotient stack $[X\times_{\Spec \bZ[P]}\Spec \bZ[\frac{1}{n}P]/\boldsymbol{\mu}_n(P)]$, where $\boldsymbol{\mu}_n(P)$ is the Cartier dual of the cokernel of the map $P^\gp\to  \frac{1}{n}P^\gp$, acting on the product $X\times_{\Spec \bZ[P]}\Spec \bZ[\frac{1}{n}P]$ via the trivial action on $X$ and the natural action on the second factor.

A limit version of this construction gives  an fpqc atlas for the infinite root stack (which is not  algebraic). Indeed, again if $X\to \Spec \bZ[P]$ is a Kato chart, and if we denote by $P_\bQ$  the rational cone spanned  by $P$ inside  $P^\gp\otimes\bQ$, then the infinite root stack $\radice{\infty}{X}$ is isomorphic to the global quotient $[X\times_{\Spec \bZ[P]}\Spec \bZ[P_\bQ]/\boldsymbol{\mu}_\infty(P)]$, where $\boldsymbol{\mu}_\infty(P)=\varprojlim_n \boldsymbol{\mu}_n(P)$, and we stackify with respect to the fpqc topology.

In general, there exists an \'etale cover $\{X_i\to X\}$ where the $X_i$ admit Kato charts $X_i\to \Spec \bZ[P_i]$, and this provides an fpqc atlas $$\bigsqcup_i X_i \times_{\Spec \bZ[P_i]}\Spec \bZ[(P_i)_\bQ]\to \radice{\infty}{X}.$$

Sheaves on an appropriately defined ``small \'etale site'' of the infinite root stack of a fine saturated log scheme $X$ are the same as sheaves on the Kummer \'etale site of $X$ (a natural generalization of the small \'etale site of a scheme), as we recall now.

\begin{definition}
We denote by $\Sh\left(\radice{\i}{X}^{rep}_{\et}\right)$ the $\i$-topos of sheaves of spaces on the site of maps $\cY \to \radice{\infty}{X}$ of stacks which are \'etale and representable by an algebraic space, with the induced \'etale topology.
\end{definition}

We recall the notion of Kummer \'etale maps of log schemes.

\begin{definition}
Let $f:X \to Y$ be a morphism of fine saturated log schemes whose underlying map of schemes is of finite presentation.
\begin{itemize}
\item  $f$ is \textbf{log \'etale} if \'etale locally on $X$ and $Y$ there exists a chart
$$\left(P \to M_X(X), Q \to M_Y(Y),P \to Q\right)$$ of $f$ such that
$$Y \to X \times_{\Spec\left(\mathbb{Z}\left[P\right]\right)} \Spec \mathbb{Z}\left[P\right]$$ and
$$\Spec(\cO_Y\left[Q^{\gp}\right]) \to \Spec(\cO_Y\left[P^{\gp}\right])$$ are \'etale maps of schemes.
\item $f$ is \textbf{Kummer} if the induced map $f^{-1} \overline{M_Y} \to  \overline{M_X}$ is injective and the cokernel of 
$$f^{-1} \overline{M_Y}^{\gp} \to \overline{M_X}^{\gp}$$ is torsion.
\item $f$ is called \textbf{Kummer \'etale} if  it is both log \'etale and Kummer.
\end{itemize}
\end{definition}

\begin{definition}
Let $X$ be a finite saturated log scheme. Its \textbf{Kummer \'etale site}, denoted by $X_{k\et}$ is the subcategory of log schemes over $X$ on the Kummer \'etale morphisms, with jointly surjective Kummer \'etale morphisms as covers.
\end{definition}

\begin{theorem}\label{thm:TVmain}
There is a canonical equivalence of $\i$-categories $$\Sh\left(X_{k\et}\right) \simeq \Sh\left(\radice{\i}{X}^{rep}_{\et}\right).$$
\end{theorem}

\begin{proof}
Since both sites have finite limits, this follows immediately from \cite[Theorem 6.21]{TV} and the proof of \cite[Proposition 6.4.5.7]{htt}.
\end{proof}

We also recall the notion of log regularity, a generalization of regularity to the logarithmic setup.

\begin{definition}[{\cite[Definition 2.2]{niziol}}]
A fine saturated locally Noetherian log  scheme $X$ is \textbf{log regular} if for every point $x\in X$ the ring $\cO_{X,\overline{x}}/I_{\overline{x}}\cO_{X,\overline{x}}$ is regular, and $${\rm dim}(\cO_{X,\overline{x}})={\rm dim}(\cO_{X,\overline{x}}/I_{\overline{x}}\cO_{X,\overline{x}})+{\rm rk}(\overline{M_X})_{\overline{x}}^\gp$$ where $\overline{x}$ is a geometric point lying over $x$, $\cO_{X,\overline{x}}$ denotes the stalk of the structure sheaf for the \'etale topology, and $I_{\overline{x}}=M_{X,\overline{x}}\setminus \cO_{X,\overline{x}}^\times$.
\end{definition}

Given a fine saturated log scheme $X,$ there are various natural ways to construct a pro-space which is a candidate for the ``underlying homotopy type'' of $X$:

\begin{itemize}
 \item[1)] the \'etale homotopy type of the $\i$-root stack $\radice{\i}{X}$,
 \item[2)] the \'etale homotopy type of the $\i$-root stack $\radice{\i}{X}^{pro}$ regarded as a pro-object,
 \item[3)] the shape of the $\i$-topos $\Sh\left(X_{k\et}\right)$,
 \item[4)] the shape of the $\i$-topos $\Sh\left(\radice{\i}{X}^{rep}_{\et}\right)$,
 \item[5)] for a log scheme locally of finite type over $\mathbb{C},$ the profinite completion of its Kato-Nakayama space $X_{log}$, 
\end{itemize}
and in the log regular case, also
\begin{itemize}
 \item[6)] the profinite \'etale homotopy type of the trivial locus $X^{\triv}.$
\end{itemize}
One of the main goals of this paper is to show that, up to profinite completion, $1)-4)$ agree. Note that \cite[Theorem 7.3]{KNIRS} implies that $2)$ and $5)$ agree when working over $\mathbb{C}.$

\begin{remark}
Note that Theorem \ref{thm:TVmain} immediately implies that $3)$ and $4)$ agree, even before profinite completion.
\end{remark}

\begin{remark}
Although $1)$ and $2)$ may seem as though they are practically the same, they are subtly different. Moreover, the proof that they agree after profinite completion is very technical and involved. Furthermore, in the proof of \cite[Theorem 4.36]{Berner} Berner implicitly assumes, without proof, that $4)$ and $2)$ gives the same profinite homotopy type. We will show that this is indeed true, but it is very far from obvious.
\end{remark}

Expanding on results of Berner \cite{Berner}, we will show furthermore in Section \ref{sec:log-regular}, that if $X$ is log regular and in characteristic zero, then $6)$ also agrees with $1)-4)$ and in arbitrary characteristic the result is still true after completing away from the residue characteristic. See Corollary \ref{cor:Berner}.

\section{fpqc descent for Betti stacks}
\label{bettidescent}
In this section we will prove a technical result needed to establish that $1)$ and $2)$ above are the same after profinite completion. Namely, we will prove that a Betti stack $\Delta^{\et}\left(V\right)$ for $V$ a $\pi$-finite space satisfies fpqc descent. We will do this over a fixed base scheme $S.$

\begin{proposition}\label{prop:Gfin}
Let $G$ be a finite group. Then its classifying stack of (\'etale) torsors $BG$ satisfies fpqc descent.
\end{proposition}

\begin{proof}
This follows from the well known fact that any algebraic stack with quasi-affine diagonal satisfies fpqc descent.
\end{proof}

Let $X$ be a scheme. There is a canonical geometric morphism $$\lambda:\Sh\left(X_{\et}\right) \to \Sh\left(\Sch/X,\et\right).$$ The inverse image functor $\lambda^*$ is given simply by restriction. It in fact has a \emph{left} adjoint $\lambda_!.$ Identifying $\Sh\left(X_{\et}\right)$ with the $\i$-category $\mathfrak{DM}^{\et}_X$ of $\i$-Deligne Mumford stacks \'etale over $X,$ and using the canonical identification $$\Sh\left(\Sch/X,\et\right) \simeq \Sh\left(\Sch,\et\right)/X,$$ $\lambda_!$ sends an \'etale map $\cY \to X$ from a Deligne-Mumford stack simply to itself, as an object of $\Sh\left(\Sch,\et\right)/X.$ Since \'etale maps are stable under pullback, we conclude that $\lambda_!$ preserves fibered products. It moreover preserves the terminal object, and therefore is left exact. We conclude that there is a well-defined geometric morphism in the opposite direction
$$\tau:\Sh\left(\Sch/X,\et\right) \to \Sh\left(X_{\et}\right)$$
with $\tau^*=\lambda_!$ and $\tau_*=\lambda^*.$

Unwinding definitions, we see that if $G$ is a sheaf on the small \'etale site of $X,$ that $\tau^*\left(G\right)$ is the sheaf
$$\left(f:Y \to X\right) \mapsto \Gamma_Y\left(f^*G\right).$$

\begin{lemma} \label{lem:Bhatt}
Let $X$ be an affine scheme. Let $A$ be a torsion abelian sheaf on the small \'etale site of $X,$ and 
$$K\left(\tau^*\left(A\right),n\right):\left(\Sch/X\right)^{op} \to \Spc$$ the \'etale sheaf of spaces given by the Eilenberg-Maclane construction applied to $\tau^*\left(A\right).$ Then $K\left(\tau^*\left(A\right),n\right)$ satisfies fpqc descent.
\end{lemma}

\begin{proof}
Consider the functor 
\begin{eqnarray*}
\cF_A:\left(\Sch/X\right)^{op} &\to& \mathbf{Ch}_{\ge 0}\\
\left(f:Y \to X\right) &\mapsto& R\Gamma\left(Y_{\et},f^*A\right)
\end{eqnarray*}
where $R\Gamma\left(Y_{\et},f^*A\right)$ is global sections over $Y$ of the right derived functor of the abelian sheaf $f^*A$ on $Y_{\et}.$ By \cite[Proposition 5.1]{BhattMathew}, this functor is a sheaf with respect to the $v$-topology. Since every standard fpqc cover is a $v$-cover, we conclude it is also a sheaf with respect to the fpqc topology. The classical Dold-Kan correspondence furnishes us with an equivalence of $\i$-categories $$DK:\mathbf{Ch}_{\ge 0} \simeq \left(\mathbf{Ab}^{\Delta^{op}}\right)^o,$$ between non-negatively graded chain complexes, and the $\i$-category associated to the model category of simplicial abelian groups. The latter is monadic over the $\i$-category of spaces $\Spc,$ as it is the $\i$-category of algebras for an algebraic theory. The forgetful functor $u:\left(\mathbf{Ab}^{\Delta^{op}}\right)^o \to \Spc,$ sends a simplicial abelian group to its underlying simplicial set. Since this functor preserves limits, we conclude that for any $n,$ $$u \circ DK \circ \cF_A\left[n\right]$$ is an fpqc sheaf of spaces. Unwinding the definitions, we see that this functor sends $f$ above to the space $$\Gamma_{Y_{\et}}\left(K\left(f^*A,n\right)\right).$$ By \cite[Remark 6.5.1.4]{htt}, we can identify this space with $$\Gamma_{Y_{\et}}\left(f^*K\left(A,n\right)\right).$$ This can in turn be identified with the space of lifts
$$\xymatrix{& \Sh\left(X_{\et}\right)/K\left(A,n\right) \ar[d]\\ \Sh\left(Y_{\et}\right) \ar@{-->}[ru] \ar[r]^-{f} & \Sh\left(X_{\et}\right).  }$$ Since $\Sh\left(X_{\et}\right)/K\left(A,n\right) \to \Sh\left(X_{\et}\right)$ is an \'etale geometric morphism, such a lift must be a map of strictly Henselian ringed $\i$-topoi, i.e. of Deligne-Mumford stacks (viewed as $\i$-topoi with a structure sheaf). Therefore, $$u \circ DK \circ \cF_A\left[n\right] \simeq \tau^*\left(K\left(A,n\right)\right)$$ which is equivalent to $K\left(\tau^*A,n\right),$ again by \cite[Remark 6.5.1.4]{htt}.
\end{proof}

\begin{corollary}
Let $A$ be a finite abelian group. Then the Betti stack $\Delta^{\et}\left(K\left(A,n\right)\right)$ satisfies fpqc descent.
\end{corollary}

\begin{proof}
First note that $\Delta^{\et}\left(K\left(A,n\right)\right)$ satisfies fpqc descent on $\Sch_S$ if and only if for all $g:X \to S$ with $X$ affine, $g^*\Delta^{\et}\left(K\left(A,n\right)\right)$ satisfies fpqc descent on $\Sch_S/X\simeq \Sch/X.$ But $$g^*\Delta^{\et}\left(K\left(A,n\right)\right) \simeq K\left(g^*\Delta^{\et}\left(A\right),n\right)$$ by \cite[Remark 6.5.1.4]{htt}, so we are done by Lemma \ref{lem:Bhatt}.
\end{proof}

Denote by $\Spc^\times$ the maximal Kan subcomplex of the $\i$-category $\Spc,$ i.e. the (large) $\i$-groupoid of spaces and equivalences. Fix an abelian group $A.$ Denote by $B\Aut\left(K\left(A,n\right)\right)$ the full subcategory of $\Spc^\times$ spanned by the single object $K\left(A,n\right).$ This is a small $\i$-groupoid, and hence canonically identified with a space in $\Spc.$ (Concretely it is the space of self homotopy equivalences of $K\left(A,n\right).$)

\begin{lemma} \label{lem:BAut}
For any finite abelian group $A,$ $\Delta^{\et}\left(B\Aut\left(K\left(A,n\right)\right)\right)$ satisfies fpqc descent.
\end{lemma}

\begin{proof}
Fix a cardinality bound on $\Sch$ so that we can work within $\i$-topoi and so that fpqc-sheafification exists. For any \'etale sheaf $H,$ denote by $aH$ its fpqc-sheafification, viewed as an object of $\Sh\left(\Sch,\et\right).$ It suffices to prove that for all affine schemes, the canonical map $$\Map\left(X,\Delta^{\et}\left(B\Aut\left(K\left(A,n\right)\right)\right)\right) \to \Map\left(X,a\Delta^{\et}\left(B\Aut\left(K\left(A,n\right)\right)\right)\right)$$ is an equivalence. Note that by \cite[p. 31]{etalehomotopy}, it follows that there is a canonical map $$\psi:B\Aut\left(K\left(A,n\right)\right) \to B\Aut\left(A\right)$$ identifying $B\Aut\left(A\right)$ with the $1$-truncation. Notice that by Proposition \ref{prop:Gfin}, we have that $$a\Delta^{\et}\left(B\Aut\left(A\right)\right) \simeq \Delta^{\et}\left(B\Aut\left(A\right)\right).$$ Therefore, it suffices to prove that for any $$\omega:* \to \Map\left(X,\Delta^{\et}\left(B\Aut\left(A\right)\right)\right) \simeq \Map\left(X,a\Delta^{\et}\left(B\Aut\left(A\right)\right)\right),$$ the induced map between the homotopy fibers of $\Map\left(X,\Delta^{\et}\left(\psi\right)\right)$ and $\Map\left(X,a\Delta^{\et}\left(\psi\right)\right)$ is an equivalence. Denote the homotopy fiber of the former by $F^{\et}_\omega,$ and the latter by $F^{\mbox{fpqc}}_{\omega}.$ By \cite[Proposition 5.5.5.12]{htt}, we have a canonical identification
$$F^{\et}_\omega \simeq \Map_{\Sh\left(\Sch,\et\right)/\Delta^{\et}\left(B\Aut\left(A\right)\right)}\left(\omega,\Delta^{\et}\left(\psi\right)\right).$$ 
Consider the map $$\omega:X \to \Delta^{\et}\left(B\Aut\left(A\right)\right).$$ It classifies a locally constant \'etale sheaf $\cF_{\omega}$ of abelian groups on $X$ with coefficients in $A.$ Explicitly, the canonical functor $$B\Aut\left(A\right) \to \mathbf{Ab}$$ corresponds to an abelian sheaf $\cF_A$ on $\Spc/B\Aut\left(A\right)$ and $\omega$ corresponds to a geometric morphism $$\overline{\omega}:\Sh\left(\Sch,\et\right)/X \to \Spc/B\Aut\left(A\right),$$ and $$\cF_{\omega}=\overline{\omega}^*\cF_A.$$ Moreover, by \cite[Theorem 2.40]{etalehomotopy}, it follows that 
$$\cF_{\omega} \simeq \tau^*\cF_{\omega'},$$ for $\cF_{\omega'}$ an abelian sheaf on the small \'etale site of $X.$ By the proof of \cite[Proposition 4.11]{etalehomotopy}, we have a canonical identification $$F^{\et}_\omega \simeq \Gamma_X\left(K\left(\cF_{\omega},n+1\right)\right).$$ By an analogous argument, we have $$F^{\mbox{fpqc}}_\omega \simeq \Gamma_X\left(aK\left(\cF_{\omega},n+1\right)\right).$$ Since $\cF_{\omega}$ is the pull back from the small \'etale topos of a torsion abelian sheaf, the result now follows from Lemma \ref{lem:Bhatt}.
\end{proof}

\begin{proposition}\label{prop:connV}
For any connected $\pi$-finite space $V,$ the Betti stack $\Delta^{\et}\left(V\right)$ satisfies fpqc descent.
\end{proposition}

\begin{proof}
We prove this by induction on homotopy dimension. Suppose that the Betti stack of all connected $\left(n-1\right)$-truncated $\pi$-finite spaces satisfies fpqc descent. Let $V$ be an $n$-truncated connected $\pi$-finite space. Then there is a Cartesian diagram in $\Spc$
$$\xymatrix{V  \ar[d] \ar[r] & B\Aut\left(\pi_n\left(V\right)\right) \ar[d]\\
V_{n-1} \ar[r] & B\Aut\left(K\left(\pi_n\left(V\right),n\right)\right)}$$
where $V_{n-1}$ is the $\left(n-1\right)$-truncation of $V.$ Since $\Delta^{\et}$ preserves finite limits, we have $$\Delta^{\et}\left(V\right) \simeq \Delta^{\et}\left(V_{n-1}\right) \times_{\Delta^{\et}\left(B\Aut\left(K\left(\pi_n\left(V\right)\right)\right)\right)} \Delta^{\et}\left(B\Aut\left(\pi_n\left(V\right)\right)\right).$$ Since fpqc-sheaves are stable under limits, the result now follows from the inductive hypothesis, Proposition \ref{prop:Gfin}, and Lemma \ref{lem:BAut}.
\end{proof}

\begin{definition}
Let $\sC$ denote the following Grothendieck pre-topology on $\Sch.$ A finite family of maps $$\left(f_i:X_i \to X\right)_{i=1}^{n}$$ is a \textbf{$\sC$-cover} if the canonical map $$\underset{i}\coprod X_i \to X$$ is an isomorphism. We call the corresponding Grothendieck topology the \textbf{coproduct topology}. Denote the corresponding sheafification functor by $c.$
\end{definition}

\begin{remark}\label{rmk:clopen}
Note that every $\sC$-cover can be written as a finite composition of $2$-term $\sC$-covers. Moreover, by \cite[TAG 02WN]{stacks-project}, every $2$-term $\sC$-cover is isomorphic to the inclusion of two complimentary closed and open subschemes.
\end{remark}

\begin{lemma}\label{lem:coprodsh1}
Let $J$ be a Grothendieck topology on $\Sch$ for which every $\sC$-cover is a cover, and for which open closed inclusions have $J$-descent. Let $F_1$ and $F_2$ be two $J$-sheaves of spaces. Then the coproduct of $F_1$ and $F_2$ as $\sC$-sheaves and as $J$-sheaves coincide.
\end{lemma}

\begin{proof}
Since any $J$-sheaf is a $\sC$-sheaf, if $F_1 \coprod F_2$ denotes the coproduct in presheaves, it suffices to prove that the $\sC$-sheafification $c\left(F_1 \coprod F_2\right)$ is a $J$-sheaf. Since sheafification is a transfinite composition of the plus-construction, it suffices to prove that the $\sC$-plus construction $\left(F_1 \coprod F_2\right)^+$ is a $J$-sheaf (since then it is also a $c$-sheaf and coincides with the $c$-sheafification). Consider the $J$-sheaf $OpCl$ which assigns $X$ the set of clopen subschemes of $X.$ By Remark \ref{rmk:clopen}, we have a canonical identification
$$\left(F_1 \coprod F_2\right)^+\left(X\right) \simeq \underset{Z \in OpCl\left(X\right)} \coprod \left(F_1\left(Z\right) \coprod F_2\left(X-Z\right)\right).$$ Suppose that $\left(f_\alpha:U_\alpha \to U\right)_\alpha$ is a $J$-cover of a scheme $U.$ It suffices to show that the canonical map
$$\left(F_1 \coprod F_2\right)^+\left(U\right) \to \lim \left[\underset{\alpha}\prod \left(F_1 \coprod F_2\right)^+\left(U_\alpha\right) \rightrightarrows \underset{\alpha,\beta} \prod \left(F_1 \coprod F_2\right)^+\left(U_{\alpha\beta}\right) \rrrarrow \cdots\right]$$ is an equivalence.

Note that we have a canonical equivalence between 
$$
\mathcal{L}_1:= \lim \left[\underset{\alpha}\prod \left(F_1 \coprod F_2\right)^+\left(U_\alpha\right) \rightrightarrows \underset{\alpha,\beta} \prod \left(F_1 \coprod F_2\right)^+\left(U_{\alpha\beta}\right) \rrrarrow \cdots\right]
$$
and 
\[
\resizebox{\displaywidth}{!}{\xymatrix{
\mathcal{L}_2 := \lim \left[\underset{\alpha}\prod \left(\underset{Z \in OpCl\left(U_\alpha\right)} \coprod  F_1\left(Z\right) \coprod F_2\left(U_\alpha-Z \right)\right) 
\rightrightarrows \underset{\alpha,\beta} \prod \left(
 \underset{Z \in OpCl\left(U_{\alpha\beta}\right)} \coprod 
 F_1 \left(Z \right) \coprod F_2\left(U_{\alpha\beta}-Z  \right)
\right) 
\rrrarrow \cdots
\right]. 
}}
\] 
Further we have a canonical map 
%
$$
\mathcal{L}_2 \to \underset{\alpha}\prod \left(\underset{Z \in OpCl\left(U_\alpha\right)} \coprod F_1\left(Z \right) \coprod F_2\left(U_\alpha-Z \right)\right)
$$
By composition, we obtain a map
$$
\mathcal{L}_1 \stackrel{\simeq}{\longrightarrow} \mathcal{L}_2 \to 
\underset{\alpha}\prod \left(\underset{Z \in OpCl\left(U_\alpha\right)} \coprod F_1\left(Z \right) \coprod F_2\left(U_\alpha-Z \right)\right)
$$
If $\left(s_\alpha\right)_\alpha$ is an object in the image of the map, without loss of generality we have that there is $\alpha$-indexed set $\left(Z_\alpha \subset U_\alpha\right)_\alpha$ of clopen subschemes such that $s_\alpha \in F_1\left(Z_\alpha\right)$ for all $\alpha,$ and moreover, we have that for each projection $$pr_{\alpha}:U_{\alpha\beta} \to U_{\alpha},$$ for each $\alpha$ and $\beta$ $$pr_{\alpha}^{-1}\left(Z_{\alpha}\right)=pr_{\beta}^{-1}\left(Z_{\beta}\right).$$ Since $OpCl$ is a $J$-sheaf, there is a unique clopen subscheme $Z \subset U$ such that $f_\alpha^{-1}\left(Z\right)=Z_\alpha$ for all $\alpha.$ It follows that we have a canonical equivalence between $\mathcal{L}_2$ and 
\[\resizebox{\displaywidth}{!}
{\xymatrix{ \underset{Z \in OpCl \left(U \right) } 
\coprod  
\left(\lim \left[ \underset{\alpha} \prod F_1\left(f_\alpha^{-1}\left(Z\right)\right) \rightrightarrows \underset{\alpha,\beta} \prod F_1\left(f_{\alpha\beta}^{-1}\left(Z\right)\right) \rrrarrow \cdots \right] 
\coprod \lim \left[ \underset{\alpha} \prod F_2\left(f_\alpha^{-1}\left(U-Z\right)\right)
 \rightrightarrows \underset{\alpha,\beta} \prod 
F_2\left(f_{\alpha\beta}^{-1}\left(U-Z\right)\right) \rrrarrow \cdots 
\right]
\right). 
}}
\]
Notice that $\left(f_\alpha^{-1}\left(Z\right) \to Z\right)_\alpha$ and $\left(f_\alpha^{-1}\left(U-Z\right) \to U-Z\right)_\alpha$ are $J$-covers of $Z$ and $U-Z$ respectively. Since $F_1$ and $F_2$ are $J$-sheaves, we finally have a canonical equivalence between 
\[\resizebox{\displaywidth}{!}{\xymatrix{
\underset{Z \in OpCl\left(U\right)} \coprod \left(\lim \left[ \underset{\alpha} \prod F_1\left(f_\alpha^{-1}\left(Z\right)\right) \rightrightarrows \underset{\alpha,\beta} \prod F_1\left(f_{\alpha\beta}^{-1}\left(Z\right)\right) \rrrarrow \cdots \right] \coprod \lim \left[ \underset{\alpha} \prod F_2\left(f_\alpha^{-1}\left(U-Z\right)\right) \rightrightarrows \underset{\alpha,\beta} \prod F_2\left(f_{\alpha\beta}^{-1}\left(U-Z\right)\right) \rrrarrow \cdots \right]\right) 
}}\]
and 
$$
\underset{Z \in OpCl\left(U\right)} \coprod \left(F_1\left(Z\right) \coprod F_2\left(U-Z\right)\right) = 
\left(F_1 \coprod F_2\right)^+\left(U\right).
$$
\end{proof}

\begin{corollary}\label{cor:coprodsh2}
If $F_1$ and $F_2$ are fpqc sheaves of spaces on $\Sch,$ then their coproduct as fpqc sheaves and \'etale sheaves coincide.
\end{corollary}

\begin{proof}
Since $OpCl$ is an fpqc sheaf, and every $\sC$ cover is an \'etale cover, this follows immediately from Lemma \ref{lem:coprodsh1}.
\end{proof}

\begin{theorem}\label{thm:Betti-fpqc}
Let $V$ be a $\pi$-finite space. Then its Betti stack $\Delta^{\et}\left(V\right)$ satisfies fpqc descent.
\end{theorem}
\begin{proof}
Since $\Delta^{\et}\left(V\right)$ is a finite coproduct of Betti stacks of connected $\pi$-finite spaces, this follows immediately from Proposition \ref{prop:connV} and Corollary \ref{cor:coprodsh2}.
\end{proof}

\section{Shape comparison}
\subsection{To pro or not to pro}
\label{shapecomparison1}
Fix a fine saturated log scheme $X$ over a base scheme $S.$ In this section we prove that the \'etale homotopy type of the infinite root stack of a log scheme viewed as a pro-object is the same as when it is viewed as an actual fpqc-stack, after profinite completion. 

It turns out that this statement is local in the \'etale topology of $X$, so we can first assume that $X$ is affine and has a  global Kato chart  $X\to \Spec \bZ[P]$ where $P$ is a fine saturated monoid. In this case each root stack (including the limit) is a global quotient of affine schemes. 

We write $\Sch_S$ for schemes over $S.$ Recall that by an affine $S$-scheme, we mean an $S$-scheme of the form $S \times \Spec\left(A\right) \to S.$

\begin{definition}
A functor $F:\Sch_S^{op} \to \Spc$ is \textbf{limit-preserving} if for all cofiltered limits $X=\underset{n} \lim X_n$ of {affine} $S$-schemes, the canonical map $$\underset{n} \colim F\left(X_n\right) \to F\left(X\right)$$ is an equivalence.
\end{definition}

\begin{lemma}\label{lem:limp1}
Suppose $F:\Sch_S^{op} \to \Spc$ is a colimit in $\Psh\left(\Sch_S\right)$ of algebraic spaces locally of finite type over $S,$ then $F$ is limit preserving.
\end{lemma}

\begin{proof}
By \cite[Tag 0CMX]{stacks-project}, any algebraic stack locally of finite type over $S$ is limit preserving. However, the object-wise colimit of limit-preserving functors clearly is limit preserving, since colimits commute with colimits.
\end{proof}

\begin{proposition}\label{prop:limp2}
Let $V$ be any space, then its Betti stack $\Delta^{\et}\left(V\right)$ is limit preserving.
\end{proposition}

\begin{proof}
Choose a simplicial set $V^\bullet:\Delta^{op} \to \Set$ corresponding to $V.$ Then $$\Delta^{\et}\left(V\right)=\underset{n \in \Delta^{op}} \colim \Delta^{\et}\left(V^n\right).$$ Moreover, we have that $$\underset{n \in \Delta^{op}} \colim \Delta^{\et}\left(V^n\right) \simeq a\left(\underset{n \in \Delta^{op}} \colim \Delta\left(V^n\right)\right),$$ where $\Delta\left(V^n\right)$ is the constant presheaf, and $a$ denotes \'etale sheafification. However, for any set $A,$ $\Delta^{\et}\left(A\right)\cong \underset{a \in A} \coprod 1$ is a scheme locally of finite type, and we can also write $$a\left(\underset{n \in \Delta^{op}} \colim \Delta\left(V^n\right)\right) \simeq a\left(\underset{n \in \Delta^{op}} \colim \Delta^{\et}\left(V^n\right)\right),$$ where the colimit is taken in presheaves. The presheaf $$\underset{n \in \Delta^{op}} \colim \Delta^{\et}\left(V^n\right)$$ is limit-preserving by Lemma \ref{lem:limp1}. Moreover, it follows from \cite[TAG 049N]{stacks-project} that \'etale sheafification preserves the property of being limit-preserving, so we are done.
\end{proof}

The following corollary is immediate:

\begin{corollary}\label{cor:etaleaffine}
Let $Z = \underset{\alpha} \lim Z_\alpha$ be a cofiltered limit of affine $S$-schemes. Then
$$\Pi^{\et}_\i\left(Z\right) \simeq \underset{\alpha} \lim \Pi^{\et}_\i\left(Z_\alpha\right).$$
\end{corollary}

\begin{proposition}\label{prop:limp3}
Suppose that $X$ is a finite saturated affine log scheme over $S$ with a global Kato chart $X\to \Spec\bZ[P]$. Let $\cX$ be a limit preserving fpqc sheaf of spaces which is $k$-truncated for some $k<\infty$ Then the natural map
\[
\underset{n}\colim \Map\left(\radice{n}{X}, \cX\right)\to \Map\left(\radice{\infty}{X}, \cX\right)
\]
is an  equivalence of spaces.
\end{proposition}

\begin{proof}
For every $n\in \bN$, let $U_n$ denote the fibered product $X\times_{ \Spec \bZ[P]} \Spec \bZ[\frac{1}{n}P]$, and let $U_\infty$ be the inverse limit, that is isomorphic to $X\times_{ \Spec \bZ[P]} \Spec \mathbb{Z}\left[P_{\bQ}\right]$. Moreover, let $G_n=\boldsymbol{\mu}_n(P)$
denote the Cartier dual of the cokernel of the natural inclusion $P\to\frac{1}{n}P$, and let $G_\infty=\boldsymbol{\mu}_\infty(P)$
be the inverse limit. Recall that we have equivalences
\[
\radice{n}{X}\simeq \left[U_n/G_n\right]
\]
for the natural action of $G_n$ on $U_n$ (including $n=\infty$, if we use the fpqc topology), and these are compatible with the projections between root stacks and the quotient stacks.
Notice that this implies that in the $\left(k+1,1\right)$-category of fpqc sheaves of $k$-groupoids,
$$\radice{n}{X} \simeq \underset{l \in \Delta_{\le k}^{op}} \colim \left(U_{n} \times G_n^{l}\right).$$ Note that this is a finite limit, so it commutes with filtered colimits.
Hence, we have the following string of natural equivalences
\begin{eqnarray*}
\Map\left(\radice{\infty}{X}, \cX\right) &\simeq& \underset{m \in \Delta_{\le k}} \lim \Map\left(U_\i \times G_\i^m,\cX\right)\\
&\simeq& \underset{m \in \Delta_{\le k}} \lim \cX\left(U_\i \times G_\i^m\right)\\
&\simeq& \underset{m \in \Delta_{\le k}} \lim \cX\left(\underset{n} \lim \left(U_n \times G_n^m\right)\right)\\
&\simeq& \underset{m \in \Delta_{\le k}} \lim \underset{n} \colim \cX\left(U_n \times G_n^m\right)\\
&\simeq& \underset{n} \colim \underset{m \in \Delta_{\le k}} \lim \cX\left(U_n \times G_n^m\right)\\
&\simeq& \underset{n} \colim \underset{m \in \Delta_{\le k}} \lim \Map\left(U_n \times G_n^m,\cX\right)\\
&\simeq& \underset{n} \colim \Map\left(\underset{m \in \Delta_{\le k}^{op}} \colim \left(U_n \times G_n^m\right),\cX\right)\\
&\simeq& \underset{n} \colim \Map\left(\radice{n}{X},\cX\right).\\
\end{eqnarray*}
\end{proof}

\begin{lemma}\label{lem:orproaff}
Suppose that $X$ is a finite saturated affine log scheme over $S$ with a global Kato chart. Let $\radice{\i}{X}_{pro}$ denote the infinite root stack viewed as a pro-object and $\radice{\i}{X}$ the actual limit of this pro-object, which is an fpqc-stack of groupoids. Then the canonical map
$$\widehat{\Pi}^{\et}_\i\left(\radice{\i}{X}_{pro}\right) \to \widehat{\Pi}^{\et}_\i\left(\radice{\i}{X}\right)$$ between their \underline{profinitely completed} \'etale homotopy types is an equivalence.
\end{lemma}

\begin{proof}
By definition, $\widehat{\Pi}^{\et}_\i\left(\radice{\i}{X}_{pro}\right)$ is the cofiltered limit in profinite spaces $$\underset{n} \lim \widehat{\Pi}^{\et}_\i\left(\radice{n}{X}\right).$$ Recall that the $\i$-category $\Profs$ is the opposite of the full-subcategory of $\Fun\left(\Spc^{\pi},\Spc\right),$ on those functors $$\Spc^{\pi} \to \Spc$$ which are accessible and preserve finite limits, where $\Spc^{\pi}$ is the $\i$-category of $\pi$-finite spaces. In particular, filtered limits in $\Profs$ are computed as the filtered colimit in the functor category. So we have, as a functor, for all $\pi$-finite spaces $V,$
\begin{eqnarray*}
\widehat{\Pi}^{\et}_\i\left(\radice{\i}{X}_{pro}\right)\left(V\right) &\simeq& \underset{n} \colim \widehat{\Pi}^{\et}_\i\left(\radice{n}{X}\right)\left(V\right)\\
&\simeq& \underset{n} \colim \Pi^{\et}_\i\left(\radice{n}{X}\right)\left(V\right)
\end{eqnarray*}
But, by \cite[Theorem 2.40]{etalehomotopy}, we have 
$$\Pi^{\et}_\i\left(\radice{n}{X}\right)\left(V\right) \simeq \Map_{\Sh\left(\Sch_S,\et\right)}\left(\radice{n}{X},\Delta^{\et}\left(V\right)\right),$$ for all $n$ including $n=\i.$ But by Theorem \ref{thm:Betti-fpqc} and Proposition \ref{prop:limp2}, $\Delta^{\et}\left(V\right)$ is a limit-preserving fpqc-sheaf, and since every $\pi$-finite spaces is $k$-truncated for some $k < \i,$ the result now follows from Proposition \ref{prop:limp3}.
\end{proof}

\begin{theorem}\label{thm:prothesame}
Let $X$ be a fine saturated log scheme over $S$. Let $\radice{\i}{X}_{pro}$ denote the infinite root stack viewed as a pro-object and $\radice{\i}{X}$ the actual limit of this pro-object, which is an fpqc-stack of groupoids. Then the canonical map
$$\widehat{\Pi}^{\et}_\i\left(\radice{\i}{X}_{pro}\right) \to \widehat{\Pi}^{\et}_\i\left(\radice{\i}{X}\right)$$ between their \underline{profinitely completed} \'etale homotopy types is an equivalence.
\end{theorem}

\begin{proof}
Firstly, notice that since we are working with the profinitely completed \'etale homotopy type, we may work with hypersheaves rather than \v{C}ech sheaves, since the difference between the shape of an $\i$-topos and its hypercompletion is erased by profinite completion. Choose an \'etale hypercover $$U_\bullet:\Delta^{op} \to \mathbb{H}\!\operatorname{yp}\!\Sh\left(\Sch_S\right)/X$$ of $X$ such that each $U_k$ is the coproduct of fine saturated affine log schemes each of which admits a global Kato chart
$$U_k=\underset{\alpha \in I_k} \coprod X_{\alpha,k}.$$
Consider the pro-object $p:\radice{\i}{X}_{pro} \to X,$ which may be viewed as a functor $$\mathbb{N} \to \mathbb{H}\!\operatorname{yp}\Sh\left(\Sch_S\right)/X.$$ Then, for each $n$ including $n=\i$ $$\left(p_n\right)^*U_{\bullet}$$ is a hypercover of $\radice{n}{X}.$  By a completely analogous argument to the proof of \cite[Lemma 6.1]{KNIRS}, we have that $$\widehat{\Pi}^{\et}_\i\left(\radice{\i}{X}_{pro}\right) \simeq \underset{k \in \Delta^{op}} \colim \left[\widehat{\Pi}^{\et}_\i\left(\underset{n} \lim p_n^*U_k\right)\right].$$
Notice however that for each $n$ and $k$ we have canonical identifications
$$\left(p_n\right)^*\left(U_k\right) \simeq \underset{\alpha  \in I_k} \coprod \left(\radice{n}{X} \times_{X} X_{\alpha,k}\right) \simeq \underset{\alpha  \in I_k} \coprod \radice{n}{X_{\alpha,k}}.$$
By Corollary \ref{cor:fixed2}, we therefore have that for each $k,$ $$\widehat{\Pi}^{\et}_\i\left(\underset{n} \lim p_n^*U^k\right) \simeq \underset{\alpha  \in I_k} \coprod \widehat{\Pi}^{\et}_\i\left(\radice{\i}{X_{\alpha,k}}_{pro}\right).$$
Since each $X_{\alpha,k}$ is affine with a global Kato chart, by Lemma \ref{lem:orproaff}, we have for each $k$ and each $\alpha$ 
$$\widehat{\Pi}^{\et}_\i\left(\radice{\i}{X_{\alpha,k}}_{pro}\right) \simeq \widehat{\Pi}^{\et}_\i\left(\radice{\i}{X_{\alpha,k}}\right)$$ and hence
\begin{eqnarray*}
 \widehat{\Pi}^{\et}_\i\left(\radice{\i}{X}_{pro}\right) &\simeq& \underset{k \in \Delta^{op}} \colim \underset{\alpha \in I_k} \coprod \widehat{\Pi}^{\et}_\i\left(\radice{\i}{X_{\alpha,k}}\right)\\
 &\simeq& \underset{k \in \Delta^{op}} \colim \underset{\alpha \in I_k} \coprod \widehat{\Pi}^{\et}_\i\left(\radice{\i}{X} \times_{X} X_{\alpha,k}\right)\\
 &\simeq& \widehat{\Pi}^{\et}_\i\left(\underset{k \in \Delta^{op}} \colim \underset{\alpha \in I_k} \coprod \left(\radice{\i}{X} \times_{X} X_{\alpha,k}\right)\right)\\
 &\simeq& \widehat{\Pi}^{\et}_\i\left(\underset{k \in \Delta^{op}} \colim \underset{\alpha \in I_k} \coprod \left(p_\i\right)^*\left(U_k\right)\right)\\
 &\simeq& \widehat{\Pi}^{\et}_\i\left(\radice{\i}{X}\right).
\end{eqnarray*}
\end{proof}

\subsection{Shape comparison with the Kummer \'etale topos}
\label{shapecomparison2}
Let us take $\Aff'$ to be a suitable \emph{small} subcategory of affine schemes, closed under finite limits and \'etale morphisms, and containing the empty scheme.





\begin{definition}
Let $\cX$ be a sheaf of groupoids on $\left(\Aff',\et\right)$ with affine diagonal. Denote by $\cX^{\mathfrak{DM}_{rep}}_{\et}$ the full subcategory of $\Sh\left(\Aff',\et\right)/\cX$ on those maps $\cY \to \cX$ such that for any map $f:T \to \cX$ from a scheme, $\cY\times_{\cX} T \to T$ is representable by a higher Deligne-Mumford stack \'etale over $T.$ 
\end{definition}

\begin{lemma}\label{lem:limittop}
For all $\cX$ as above, there is a canonical equivalence of $\i$-categories
$$\cX^{\mathfrak{DM}_{rep}}_{\et} \simeq \underset{T \to \cX} \lim T^{\mathfrak{DM}_{rep}}_{\et},$$
where the limit ranges over $\Aff'/\cX.$ 
\end{lemma}

\begin{proof}
Consider the functor
\begin{eqnarray*}
\chi:\Sh\left(\Aff',\et\right)&\to& \Topi\\
\cY &\mapsto& \Sh\left(\Aff',\et\right)/\cY.
\end{eqnarray*}
By \cite[Proposition 6.3.5.14]{htt}, $\chi$ preserves small colimits. Note that by the Yoneda lemma, $$\cX \simeq \underset{T \to \cX} \colim T,$$ where the colimit ranges over $\Aff'/\cX.$ Moreover, by \cite[Proposition 6.3.2.3]{htt}, this colimit is computed by taking the limits of the underlying $\i$-categories in the opposite direction, using the inverse-image functors. So, on one hand, we have
$$\Sh\left(\Aff',\et\right)/\cX \simeq \underset{T \to \cX} \lim \Sh\left(\Aff',\et\right)/T$$ in $\widehat{\mathbf{Cat}}_\i$--- the (large) $\i$-category of small $\i$-categories. On the other hand, by \cite[Corollary 3.3.3.2]{htt}, for any functor $$F:J \to  \widehat{\mathbf{Cat}}_\i,$$ its limit can be identified with the $\i$-category of Cartesian sections of $$\underset{J} \int F \to J,$$ i.e. $$\lim F \simeq \Fun^{\mathrm{Cart}}_{J}\left(J,\underset{J} \int F\right).$$ In the case that $$J=\Aff'/\cX \stackrel{F=\chi}{\longlongrightarrow} \widehat{\mathbf{Cat}}_\i,$$ the $\i$-category of Cartesian sections can be viewed informally as collections of maps $P_f \to T$ in $\Sh\left(\Aff',\et\right)$ for each map $T \to \cX$ from an affine scheme, with coherent equivalences $$P_g \simeq \varphi^* P_f,$$ for all $\varphi:g \to f$ in $\Aff'/\cX.$ In particular, such an object yields a well-defined functor $P:\Aff'/\cX \to \Sh\left(\Aff',\et\right)/\cX,$ sending each $f:T \to \cX$ to $P_f \to T \to \cX.$ Unwinding the definitions, the equivalence $$\theta:\Sh\left(\Aff',\et\right)/\cX \stackrel{\simeq}{\longrightarrow} \underset{T \to \cX} \lim \Sh\left(\Aff',\et\right)/T$$ sends an object $\cY \to \cX$ in $\Sh\left(\Aff',\et\right)/\cX$ to the collection $\left(\cY\times_{\cX} T \to T\right)_{T \to \cX}.$ As $\theta$ is an equivalence, it has an inverse $\tau.$ Notice that, for any $\cY \to \cX$ in $\Sh\left(\Aff',\et\right)/\cX,$ since colimits are universal, we have a pullback diagram
$$\xymatrix{\underset{f:T \to \cX} \colim \theta\left(\cY\right)_f \ar[r]^-{\sim} \ar[d] & \cY \ar[d]\\
\underset{f:T \to \cX} \colim T \ar[r]^-{\sim} & \cX.}$$
Since $\theta$ is essentially surjective, it follows that $\tau$ sends a collection $\left(P_f \to T\right)_{f:T \to \cX}$ to $\colim P.$ 

For each affine scheme $T,$ $T^{\mathfrak{DM}_{rep}}_{\et}$ is the subcategory of $\Sh\left(\Aff',\et\right)/T$ spanned by maps which are representable by a higher Deligne-Mumford stack \'etale over $T.$ Notice that the above functor $\theta$ restricts to a functor $$\theta:\cX^{\mathfrak{DM}_{rep}}_{\et} \to \underset{T \to \cX} \lim T^{\mathfrak{DM}_{rep}}_{\et}.$$ We claim that the above functor $\tau$ also restricts to a functor in the opposite direction. To see this, notice that for all $g:S \to \cX$ from an affine scheme, we have a pullback diagram
$$\xymatrix{\underset{T \stackrel{f}{\to} \cX} \colim\left(S \times_{\cX} P_f\right) \ar[r] \ar[d] & \underset{T \stackrel{f}{\to} \cX} \colim\left(P_f\right) \ar[d]\\
S \ar[r]^-{g} & \cX.}$$
Fix $f:T \to \cX.$ Then $h:S \times_{\cX} T \to \cX$ is in $\Aff'/\cX$ since $\cX$ has affine diagonal. It follows that 
$$S \times_{\cX} P_f \simeq P_h.$$
But also, $$P_h \simeq P_g \times_{S} \left(S \times_{\cX} T\right) \simeq P_g \times_{\cX} T.$$
Hence we have
\begin{eqnarray*}
\underset{T \stackrel{f}{\to} \cX} \colim\left(S \times_{\cX} P_f\right) &\simeq& \underset{T \stackrel{f}{\to} \cX} \colim\left(P_g \times_{\cX} T\right)\\
&\simeq& P_g \times_{\cX} \left(\underset{T \stackrel{f}{\to} \cX} \colim T\right)\\
&\simeq& P_g \times_{\cX} \cX\\
&\simeq& P_g.
\end{eqnarray*}
Therefore, the pullback along $g$ can be identified with $P_g \to S,$ which is an \'etale map from a higher Deligne-Mumford stack. 
\end{proof}

\begin{corollary}\label{cor:colimtop}
For all $\cX$ as in Lemma \ref{lem:limittop}, the $\i$-category $\cX^{\mathfrak{DM}_{rep}}_{\et}$ is an $\i$-topos, and in $\Topi,$ we have
$$\cX^{\mathfrak{DM}_{rep}}_{\et} = \underset{T \to \cX} \colim \Sh\left(T_{\et}\right),$$ where $\Sh\left(T_{\et}\right)$ is the small \'etale $\i$-topos of the affine scheme $T.$
\end{corollary}

\begin{proof}
By \cite[Proposition 6.3.2.3]{htt}, the $\i$-category $\Topi$ of $\i$-topoi has small colimits and they are computed as limits of the underlying $\i$-categories \emph{in the opposite direction,} using the inverse-image functors. Therefore, the underlying $\i$-category of the above colimit of $\i$-topoi is $$\underset{T \to \cX} \lim \Sh\left(T_{\et}\right).$$ However, for all $T,$ there is a canonical equivalence of $\i$-categories $$\Sh\left(T_{\et}\right) \simeq T^{\mathfrak{DM}_{rep}}_{\et}.$$ The result now follows from Lemma \ref{lem:limittop}.
\end{proof}

\begin{remark}
Unwinding \cite[Definition 2.27]{etalehomotopy}, we see that the $\i$-topos $\cX^{\mathfrak{DM}_{rep}}_{\et}$ is equivalent to $\Sh\left(\cX_{\et}\right),$ and hence the shape of $\cX^{\mathfrak{DM}_{rep}}_{\et}$ is precisely the \'etale homotopy type of $\cX.$
\end{remark}

\begin{definition}
Let $\cX^{rep}_{\et}$ denote the subcategory of $\cX^{\mathfrak{DM}_{rep}}_{\et}$ spanned by those \'etale maps $\cY \to \cX$ which are representable by algebraic spaces.
\end{definition}

\begin{remark}
By the proof of Lemma \ref{lem:limittop}, $\cX^{rep}_{\et}$ is precisely the subcategory of $0$-truncated objects.
\end{remark}

\begin{lemma}\label{lem:nhyper}
Let $G$ be an $n$-truncated object in $\cX^{\mathfrak{DM}_{rep}}_{\et}$ for $n< \infty,$ then there exists a simplicial object $$\Delta^{op} \to \cX^{rep}_{\et}/G$$ such that the induced functor $$\Delta^{op} \to \cX^{\mathfrak{DM}_{rep}}_{\et}/G$$ is a hypercover.
\end{lemma}

\begin{proof}
For each $T \in \Aff',$ and for each $t \in \pi_0\left(\cX\left(T\right)\right),$ choose a map $f_t:T \to \cX$ such that $\pi_0\left(f\right)=t.$ Let $$W=\underset{t \in \pi_0\left(\cX\left(T\right)\right),T \in \Aff'} \coprod T.$$ Then $W$ is a scheme and there is a canonical map $\rho:W \to \cX.$ Moreover, every map from an affine scheme $T \in \Aff'$ factors through $\rho$ up to equivalence. Then $G \times_\cX W \to W$ is an \'etale map from a Deligne-Mumford $n$-stack. We therefore can find a hypercover $U_\bullet$ of $G \times_\cX W$ over $W$ such that $U_k$ is an algebraic space \'etale over $W$ for all $k.$ For any $f:T \to \cX,$ since $f$ factors through $W$ up to equivalence, we have that 
$\lambda_f^*\left(U_\bullet\right)$ is a hypercover of $G \times_\cX T \to T$ by algebraic spaces \'etale over $T,$ where $\lambda_f:T \hookrightarrow W$ is the inclusion into the coproduct. It follows that $U_\bullet$ induces a canonical hypercover of $\theta\left(G\right)$ by $0$-truncated objects. The result now follows from Lemma \ref{lem:limittop}.
\end{proof}

\begin{corollary}\label{cor:epiref}
Any epimorphism $\cY' \to \cY$ in $\cX^{\mathfrak{DM}_{rep}}_{\et}$ can be refined by an epimorphism $U \to \cY$ with $U$ in $\cX^{rep}_{\et}.$
\end{corollary}

\begin{lemma}
Let $\cE$ be an $n$-topos for some finite $n.$ Denote by $K$ the following Grothendieck topology on $\cE$: a collection of maps $$\left(f_\alpha:E_\alpha \to E\right)_\alpha$$ is a $K$-cover if and only if the induced map $$\underset{\alpha} \coprod E_\alpha \to E$$ is an epimorphism. Then $$\cE \simeq \Sh_{n-1}\left(\cE,K\right),$$ i.e. $\cE$ is equivalent to the $n$-category of sheaves of $\left(n-1\right)$-groupoids on the large site $\left(\cE,K\right).$ (In other words, a presheaf is a $K$-sheaf if and only if it is representable).
\end{lemma}

\begin{proof}
Let $\cU$ be the Grothendieck universe of small sets. Let $\cV$ be a Grothendieck universe such that $\cU \in \cV.$ Then $\cE$ is $\cV$-small. Choose a $\cU$-small subcanonical site $\left(\sC,J\right)$ such that $\Sh_{n-1}\left(\sC,J\right) \simeq \cE,$ with $\sC$ an $n$-category. Denote by $i:\sC \hookrightarrow \cE$ the fully faithful inclusion. Then, by restriction and left Kan extension, the $n$-category of presheaves of $\cV$-small $\left(n-1\right)$-groupoids $\widehat{\Psh}_{n-1}\left(\sC\right)$ is a left exact localization of $\widehat{\Psh}_{n-1}\left(\cE\right).$ Furthermore, taking $J$-sheaves gives a further left exact localization. So, there is some Grothendieck topology $K'$ on $\cE$ such that the $n$-category of $\cV$-small $K$-sheaves of $\left(n-1\right)$-groupoids on $\cE$ is equivalent to the $n$-category of $\cV$-small $J$-sheaves of $\left(n-1\right)$-groupoids on $\sC.$ Moreover, it is easy to check that this restricts to an equivalence on the subcategories thereof on those sheaves which take values in $\cU$-small of $\left(n-1\right)$-groupoids. Hence, we get that $$\Sh_{n-1}\left(\cE,K'\right) \simeq \Sh_{n-1}\left(\sC,J\right) \simeq \cE.$$ It suffices to show that $K'=K.$ For this, it suffices to show that for a collection of maps $\left(f_\alpha:E_\alpha \to E\right)_\alpha=U,$ that the associated sieve $S_U \subset y\left(E\right)$ is a $K'$-covering sieve if and only if $$\underset{\alpha} \coprod E_\alpha \to E$$ is an epimorphism. By definition, $S_U$ is a $K'$-covering sieve if and only if its restriction to $\sC$ is a $J$-covering sieve, which is if and only if its $J$-sheafification is equivalent to the restriction of $y\left(E\right)$ to $\sC.$ Notice that this is true if and only if for all $C$ in $\sC,$ and all maps $g:C \to E,$ $g^*S_U$ is a $K$-covering sieve of $C.$ But note that $g^*S_U=S_{g^*U},$ i.e. $g^*S_U$ is the sieve associated to $$\left(E_\alpha\times_E C \to C\right)_\alpha,$$ hence is a covering sieve if and only if, in presheaves, $$\underset{\alpha} \coprod i^*y\left(E_\alpha\right) \times_{i^*y\left(E\right)} y\left(C\right) \to y\left(C\right)$$ admits local sections with respect to $J,$ i.e. if and only if its sheafification is an epimorphism in $\Sh_{n-1}\left(\sC,J\right),$ but since $$\Sh_{n-1}\left(\sC,J\right) \simeq \cE,$$ this is if and only if  $$\underset{\alpha} \coprod E_\alpha\times_E C \to C,$$ is an epimorphism in $E,$ for all $C,$ which is if and only if $$\underset{\alpha} \coprod E_\alpha \to E$$ is an epimorphism.
\end{proof}

Let $i:\cX^{rep}_{\et} \hookrightarrow \cX^{\mathfrak{DM}_{rep}}_{\et}$ be the canonical inclusion. Equip $\cX^{rep}_{\et}$ with the Grothendieck topology coming from the pretopology generated by jointly epimorphic families. Denote by $\widehat{\Sh}\left(\cX^{\mathfrak{DM}_{rep}}_{\et}\right)$ the $\i$-category of sheaves on $\cX^{\mathfrak{DM}_{rep}}_{\et}$ with values in $\cV$-small spaces, i.e. limit preserving functors $$\left(\cX^{\mathfrak{DM}_{rep}}_{\et}\right)^{op} \to \widehat{\Spc}.$$
Consider the left Kan extension
$$\xymatrix@C=2.5cm{\widehat{\Psh}\left(\cX^{rep}_{\et}\right) \ar@{-->}[r]^-{\Lan_y\left(y\circ i\right)} & \widehat{\Sh}\left(\cX^{\mathfrak{DM}_{rep}}_{\et}\right) \\
\cX^{rep}_{\et} \ar[u]^-{y} \ar[ru]_-{ y\circ i} & }$$
where we are taking $\cV$-small presheaves of spaces (i.e. ``large spaces''). By \cite[Proposition 6.1.5.2]{htt}, $i_!:=\Lan_yi$ is left exact. By the Yoneda Lemma, it has a right adjoint $i^*,$ given by restriction. Since $i_!$ preserves effective epimorphisms, it follows that $i^*$ sends sheaves to sheaves. Hence there is an induced geometric morphism of $\i$-topoi in the universe $\cV$
$$\widehat{\Sh}\left(\cX^{\mathfrak{DM}_{rep}}_{\et}\right) \to \widehat{\Sh}\left(\cX^{rep}_{\et}\right).$$ 

\begin{lemma}\label{lem:bign}
The geometric morphism $\widehat{\Sh}\left(\cX^{\mathfrak{DM}_{rep}}_{\et}\right) \to \widehat{\Sh}\left(\cX^{rep}_{\et}\right)$ induces an equivalences between the subcategories of $n$-truncated objects, for all $n < \i.$
\end{lemma}

\begin{proof}
Fix $n < \i.$ The subcategories of $n$-truncated objects are $\left(n+1\right)$-topoi. Therefore,  $\tau_{\le n}\widehat{\Sh}\left(\cX^{\mathfrak{DM}_{rep}}_{\et}\right)$ can be identified with the $\left(n+1\right)$-category $$\widehat{\Sh}_n\left(\cX^{\mathfrak{DM}_{rep}}_{\et}\right)$$ of sheaves of $\cV$-small $n$-groupoids on $\cX^{\mathfrak{DM}}_{rep}$ with respect to the epimorphism topology. It follows from Corollary \ref{cor:epiref}, that for any presheaf of $n$-groupoids $\cF$ on $\cX^{\mathfrak{DM}}_{rep},$ the value of its sheafification on an object $U$ in $\cX^{rep}_{\et}$ is the same as the value of the sheafification of $i^*\cF.$ In particular, since $i^*$ is moreover left exact, it follows that if $\varphi:\cZ \to \cY$ is an effective epimorphism in $\cX^{\mathfrak{DM}}_{rep},$ then $i^*y\left(\varphi\right)$ is an effective epimorphism in $\widehat{\Sh}\left(\cX^{rep}_{\et}\right).$ Consider the functor $$i_*:\widehat{\Sh}_n\left(\cX^{rep}_{\et}\right) \to \widehat{\Psh}_n\left(\cX^{\mathfrak{DM}_{rep}}_{\et}\right)$$ defined by $$i_*\left(\cG\right)\left(\cY\right):=\Map\left(i^*y\left(\cY\right),\cG\right).$$ Then since $i^*$ preserves effective epimorphisms, $i_*$ sends sheaves to sheaves, and we abuse notation and write $$i_*:\widehat{\Sh}_n\left(\cX^{rep}_{\et}\right) \to \widehat{\Sh}_n\left(\cX^{\mathfrak{DM}_{rep}}_{\et}\right).$$ By the Yoneda Lemma $i_*$ is right adjoint to $i^*$ (restricted to $n$-truncated objects). Therefore $i^*$ preserves $\cV$-small colimits. We claim that the unit and the counit of the adjunction $i_! \dashv i^*$ are equivalences. The unit $\eta:id \to i^*i_!$ is a map in $$\Fun^L\left(\widehat{\Sh}\left(\cX^{rep}_{\et}\right),\widehat{\Sh}\left(\cX^{rep}_{\et}\right)\right)$$ so by \cite[Theorem 5.1.5.6 and Proposition 5.5.4.20]{htt} it suffices to check the component of the unit is an equivalence along representables. However, this is clear. Consider now the counit $$\epsilon:i_!i^* \to id.$$ Let $\sD$ be the subcategory of $\widehat{\Sh}\left(\cX^{\mathfrak{DM}_{rep}}_{\et}\right)$ on those objects $A$ for which $\epsilon_A$ is an equivalence. Since both functors $i_!$ and $i^*$ preserves colimits, $\sD$ is closed under colimits. It also contains the essential image of $y \circ i.$ However, Lemma \ref{lem:nhyper} implies that every representable is a colimit of an object in the essential image of $y \circ i,$ and since every sheaf is a colimit of representables, we are done, and we conclude that $i_!$ and $i^*$ restrict to equivalences on the subcategory of $n$-truncated objects.
\end{proof}

\begin{lemma}\label{lem:smalln}
Denote by $\Sh\left(\cX^{rep}_{\et}\right)$ the subcategory of  $\widehat{\Sh}\left(\cX^{rep}_{\et}\right)$ on those sheaves taking values in $\cU$-small spaces. Then $\Sh\left(\cX^{rep}_{\et}\right)$ is an $\i$-topos and there is a canonical geometric morphism $\cX^{\mathfrak{DM}_{rep}}_{\et} \to \Sh\left(\cX^{rep}_{\et}\right)$ which induces an equivalence on the full subcategory of $n$-truncated objects, for all $n< \i.$
\end{lemma}

\begin{proof}
Let $l:\sC \hookrightarrow \cX^{rep}_{\et}$ be a small subcategory of the topos $\cX^{rep}_{\et}$ which is strongly generating, and which has finite limits. Then there is a Grothendieck topology $J$ on $\sC$ such that $\Sh_J\left(\sC,\Set\right)\simeq \cX^{rep}_{\et}.$ Since $\sC$ has finite limits, and both $\sC$ and $\cX^{rep}_{\et}$ are $\cV$-small, by the proof of \cite[Proposition 6.4.5.7]{htt}, it follows that $$l^*:\widehat{\Sh}\left(\cX^{rep}_{\et}\right) \stackrel{\sim}{\longrightarrow} \widehat{\Sh}_J\left(\sC\right).$$ Since $l^*$ is an equivalence, its left adjoint $l_!$ is also its right adjoint. It follows that $l^*$ restricts to an equivalence between the full subcategories of sheaves with values in $\cU$-small groupoids, i.e. $$l^*:\Sh\left(\cX^{rep}_{\et}\right) \stackrel{\sim}{\longrightarrow} \Sh_J\left(\sC\right).$$ Hence, $\Sh\left(\cX^{rep}_{\et}\right)$ is an $\i$-topos. Consider the composite $$i \circ l:\sC \hookrightarrow \cX^{\mathfrak{DM}_{rep}}_{\et}.$$ Its left Kan extension along the Yoneda embedding is a colimit preserving left exact functor $$\Sh_J\left(\sC\right) \to \cX^{\mathfrak{DM}_{rep}}_{\et}.$$ But it can be identified with the restriction of $i_!$ to $$\Sh_J\left(\sC\right) \simeq \Sh\left(\cX^{rep}_{\et}\right) \subset \widehat{\Sh}\left(\cX^{rep}_{\et}\right).$$ The result now follows from Lemma \ref{lem:bign}.
\end{proof}

\begin{theorem}
\label{protruncated}
Let $X$ be a fine saturated log scheme. Then the pro-truncated shape of $\Sh\left(\radice{\i}{X}^{rep}_{\et}\right)$ is equivalent to the pro-truncation of the \'etale homotopy type of $\radice{\i}{X}.$
\end{theorem}

\begin{proof}
By \cite[Proposition 3.19 (d)]{TV}, $\radice{\i}{X}$ has affine diagonal. So its pro-truncated shape agrees with that of the $\i$-topos $\radice{\i}{X}^{\mathfrak{DM}_{rep}}_{\et},$ by Lemma \ref{lem:smalln}. But by Corollary \ref{cor:colimtop},
$$\radice{\i}{X}^{\mathfrak{DM}_{rep}}_{\et} \simeq \underset{T \to \radice{\i}{X}} \colim \Sh\left(T_{\et}\right),$$ and the shape of the right hand-side is by definition the \'etale homotopy type of $\radice{\i}{X}.$
\end{proof}

\begin{corollary}
The pro-truncated shape of the Kummer \'etale topos agrees with the pro-truncation of the \'etale homotopy type of the infinite root stack.
\end{corollary}

\begin{proof}
This follows immediately from Theorem \ref{thm:TVmain}.
\end{proof}

\section{Shape in the log regular case}\label{sec:log-regular}

In this section we will give yet another description of the profinite homotopy type of a fine saturated log scheme, in the log regular case. We will then use this to relate the \'etale homotopy type of $\mathbb{G}_m$ to that of $B \boldsymbol{\mu}_{\i}$, where recall that $\boldsymbol{\mu}_\infty=\varprojlim_n \boldsymbol{\mu}_n$ is the inverse limit of the group schemes of $n$-th roots of unity. We will work with a log regular locally Noetherian log scheme $X$ over a locally Noetherian base $S.$

We will be using the following recent result:

\begin{theorem}\cite[Theorem 4.34]{Berner} \label{thm:Berner}
Let $X$ be a log regular locally Noetherian log scheme. Then for any locally constant sheaf $A$
of finite abelian groups with orders invertible on $X$, the inclusion $i: X^{\triv} \to X$ induces isomorphisms
in sheaf cohomology groups
$$H^*\left(X_{k\et},A\right) \cong H^*_{\et}\left(X^{\triv},i^*A\right).$$
Moreover, $i$ induces an isomorphism for any geometric point $x \in X^{ \triv}$ on pro-$\ell$-completions for any prime number $\ell$ invertible on $X$

$$\pi_1^{k\et}\left(X,x\right)^{\wedge}_{\ell} \cong \pi_1^{\et}\left(X^{\triv},x\right)^{\wedge}_{\ell}.$$
More generally, the category of Kummer \'etale covers of $X$ is equivalent to the category of \'etale covers of $X^{\triv}$ which extend to tamely ramified covers of $X.$
\end{theorem}

\begin{lemma}
If $S=\Spec k$ for $k$ a field of characteristic zero, then any \'etale cover of $X^{\triv}$ extends to a tamely ramified cover of $X.$
\end{lemma}
\begin{proof}
 Let $D=X \setminus X^{\triv}$. Let $Y \to X^{\triv}$ be an  \'etale cover. If $X$ is log regular, the rings $\cO_{X, D_i}$ are discrete valuation rings for any irreducible component $D_i$ of $D$, and therefore we can make sense of the ramification of $Y \to X^{\triv}$ over $D$ (see \cite[Section 53.31]{stacks-project}). Further, in characteristic zero, all \'etale maps
$Y \to   X^{\triv} $ are tamely ramified over $D$: this follows directly from the definition of tame ramification. Hence by  \cite[Proposition B.7]{Hoshi}, $Y \to X^{\triv}$ extends to a Kummer \'etale map $Y' \to X$, which is in particular tamely ramified. 
\end{proof}

\begin{lemma}\label{lem:Gtorsors}
Let $\ell$ be a prime invertible on $X,$ and let $G$ be a finite group whose order is relatively prime to $\ell.$ Then  there is an equivalence of categories between $G$-torsors on $X^{\triv}$ and G-torsors in the Kummer \'etale topos. 
\end{lemma}
\begin{proof}
We use the notations from the previous proof: namely we set $D:=X \setminus X^{\triv}$ and we let 
$D_i$ be the irreducible components of $D$.  
We observe first that if $$Y \to X^{\triv}$$ is a $G$-torsor, then it is tamely ramified over $D$. Recall from \cite[0BSE]{stacks-project} that, in order to check this, we  consider the fiber product 
$$
\xymatrix{
\Spec L_{D_i} \ar[r] \ar[d] & Y \ar[d] \\
\Spec K_{X,D_i} \ar[r] & X^{\triv}
}
$$
where: 
\begin{itemize}
\item $K_{X,D_i}$ is the function field of the component containing $D_i$,
\item the map $\Spec K_{X,D_i}  \to X^{\triv}$ is given by the projection 
$$
\Spec K_{X,D_i}  \cong 
\Spec \mathcal{O}_{X, D_i}  \times_X X^{\triv} \to X^{\triv}.
$$
\end{itemize}
Now, $L_{D_i}$ is a finite $K_{X,D_i}$-algebra: it splits as a product of finite field extensions 
$$
\Spec L_{D_i}  \cong \coprod_{j \in J} \Spec L_{D_i, j} ,
$$
where $J$ is a finite set. 
Further $\Spec L_{D_i}  \to \Spec K_{X,D_i} $ is a $G$-torsor, since it is the base change of a $G$-torsor. Thus $G$ acts transitively on the components  $\Spec L_{D_i, j}$ and this implies that all components are isomorphic, and each 
map $\Spec L_{D_i, j}  \to \Spec K_{X,D_i} $ is Galois for a quotient $G'$ of $G$. By   \cite[Tag 09E3]{stacks-project} all ramification indexes 
of $\Spec L_{D_i, j}  \to \Spec K_{X,D_i} $ are equal, and thus in particular they  divide $|G'|$. Since $|G'|$ divides $|G|$, we conclude that all ramification indexes  are  coprime to $\ell$, and  this is exactly the definition of tame ramification over $D$.

By \cite[Proposition B7]{Hoshi} we have an equivalence of categories $\Phi$ between: 
\begin{enumerate}
\item the category $\mathcal{C}_1$ of \'etale coverings of $X^{\triv}$ which are tamely ramified over $D$, and
\item the category $\mathcal{C}_2$ of Kummer \'etale coverings of $X$. 
\end{enumerate} 
Note that $\cC_1$ and $\cC_2$ have finite limits. We will denote by $-\times_{\mathcal{C}_i}-$, \, $i=1,2$,  the Cartesian product in $\cC_1$ and in $\cC_2$: it corresponds  respectively to the fiber product of schemes over $X^{\triv}$, and to the fiber product of fine and saturated log schemes over $X$.

Consider the group objects 
$$\mathcal{G}_1:=X^{\triv} \times G \to X^{\triv} \in \mathcal{C}_1,  \quad \mathcal{G}_2:=X \times G \to X \in \mathcal{C}_2.$$ 
Note that $G$-torsors in $\mathcal{C}_1$ can be described purely categorically: they are objects 
$P \in \mathcal{C}_1$ with an action of $\mathcal{G}_1$ such that the natural map $P \times_{\mathcal{C}_1} \mathcal{G}_1 \to P \times_{\mathcal{C}_1} P$ is an isomorphism. The analogous description holds for Kummer \'etale $G$-torsors in $\mathcal{C}_2$, in terms of $\mathcal{G}_2$.   
Since $\Phi$ is an equivalence, it will yield an equivalence between the category of $G$-torsors in $\mathcal{C}_1$ (i.e. $G$-torsors on $X^{\triv}$ that are tamely ramified over $D$)  and the category of $G$-torsors in $\mathcal{C}_2$. We showed  above  that all $G$-torsors are in fact automatically tamely ramified over $D$, and this concludes the proof. 
\end{proof}

\begin{remark}\label{rmk:FinGp}
In characteristic zero, the above holds for any finite group $G.$
\end{remark}

\begin{corollary}\label{cor:Berner}
Let $\ell$ be a prime invertible on $X,$ then $i$ induces an equivalences
$$\Shape\left(X_{k\et}\right)^{\wedge}_{\ell} \simeq \Pi^{\et}_\i\left(X^{\triv}\right)^{\wedge}_{\ell}.$$ Moreover, in characteristic zero, there is a profinite homotopy equivalence
$$\widehat{\Shape}\left(X_{k\et}\right) \simeq \widehat{\Pi}^{\et}_\i\left(X^{\triv}\right).$$
\end{corollary}

\begin{proof}
Using Theorem \ref{thm:Berner} and Lemma \ref{lem:Gtorsors}, the proof is completely analogous to \cite[Proposition 4.11]{etalehomotopy}. The statement in characteristic zero then holds by the same argument, using Remark \ref{rmk:FinGp}.
\end{proof}

\begin{remark}
In fact, the above Corollary holds for a less drastic localization than $\ell$-completion. If $\sC$ is the smallest subcategory of $\Spc$ closed under finite limits and retracts containing the subcategories $a)-c)$ of \cite[Theorem 3.25]{reletale}, then the above holds up to $\left(\blank\right)^{\wedge}_{\sC}$-localization, by the same argument. Furthermore, the proof of the next theorem carries over for $\left(\blank\right)^{\wedge}_{\sC}$-localization as well. We expect it actually holds for completion with respect to all spaces $V$ whose Betti stack $\Delta^{\et}\left(V\right)$ in $S$-schemes is $\mathbb{A}^1$-invariant.
\end{remark}

\begin{theorem}
Let $S$ be a locally Noetherian scheme, and denote by $\boldsymbol{\mu}_\i$ the affine group scheme over $S$ $$\underset{n} \lim \boldsymbol{\mu}_n.$$ Let $\ell$ be a prime invertible on $S,$ and $N$ any non-negative integer. Then there is an equivalence of $\ell$-profinite spaces
$$\Pi^{\et}_\i\left(\mathbb{G}_m^N\right)^{\wedge}_{\ell} \simeq  \Pi^{\et}_\i\left(B\boldsymbol{\mu}_{\i}^N\right)^{\wedge}_{\ell}.$$ Moreover, in characteristic zero, this holds up to profinite completion.
\end{theorem}

\begin{remark}
\cite[Proposition 6.3]{hoyois} establishes this result in the special case that $$S=\Spec k,$$ for $k$ a separably closed field.
\end{remark}

\begin{proof}
Consider $X=\mathbb{A}^N_S$ with its canonical log structure coming from the identification $X=S \times \Spec \mathbb{Z}[P]$, with $P=\mathbb{N}^N.$ Then $X$ is $\mathbb{A}^1$-contractible. So, for $Z$ any $S$-scheme, denoting by $L_{\mathbb{A}^1}$ the left adjoint to the inclusion of $\mathbb{A}^1$-invariant \'etale sheaves of spaces into the $\i$-category of all \'etale sheaves of spaces, we have that $$L_{\mathbb{A}^1}\left(X \times Z\right) \simeq L_{\mathbb{A}^1}\left(Z\right).$$ 
Note that the Betti stack $\Delta_{K\left(\mathbb{Z}/\ell,n\right)}$ is $\mathbb{A}^1$-invariant for any $n$ by \cite[Corollary VI.4.20]{milne}. Hence
\begin{eqnarray*}
\Map\left(X \times Z,\Delta_{K\left(\mathbb{Z}/\ell,n\right)}\right) &\simeq& \Map\left(L_{\mathbb{A}^1}\left(X\times Z\right),\Delta_{K\left(\mathbb{Z}/\ell,n\right)}\right)\\
&\simeq&\Map\left(L_{\mathbb{A}^1}\left(Z\right),\Delta_{K\left(\mathbb{Z}/\ell,n\right)}\right)\\
&\simeq&\Map\left(Z,\Delta_{K\left(\mathbb{Z}/\ell,n\right)}\right).
\end{eqnarray*}
It follows from Proposition \ref{prop:ellsame} that for any $S$-scheme $Z,$ the projection map $$X \times Z \to Z$$ induces an equivalence between $\ell$-profinite \'etale homotopy types
$$\Pi^{\et}_\i\left(X \times Z\right)^{\wedge}_{\ell} \simeq \Pi^{\et}_\i\left(Z\right)^{\wedge}_{\ell}.$$
Notice that \cite[Theorem 3.25]{reletale} implies that in characteristic zero, the Betti stack of any $\pi$-finite space is $\mathbb{A}^1$-invariant, so the analogous result for profinite completion holds.

Observe now that there is an equivalence
$$\radice{n}{X} \simeq \left[X/\boldsymbol{\mu}_n^N\right].$$ 
Since $\Pi^{\et}_\i$ preserves colimits, we have
$$\Pi^{\et}_\i\left(\radice{n}{X}\right) \simeq \underset{k \in \Delta^{op}} \colim \Pi^{\et}_\i\left(\left(\boldsymbol{\mu}_n^N\right)^k \times X\right).$$
Let $V$ be a $\mathbb{A}^1$-invariant space, and assume that $V$ is $m$-truncated for some finite $m$ (e.g. an $\ell$-finite space or, in the characteristic zero case, a $\pi$-finite space).
Then, since $X$ is $\mathbb{A}^1$-contractible, we have
\begin{eqnarray*}
 \Pi^{\et}_\i\left(\radice{n}{X}\right)\left(V\right) &\simeq& \Map\left(\Pi^{\et}_\i\left(\radice{n}{X}\right),j\left(V\right)\right)\\
 &\simeq& \Map\left(\underset{k \in \Delta^{op}} \colim \Pi^{\et}_\i\left(\left(\boldsymbol{\mu}_n^N\right)^k \times X\right),j\left(V\right)\right)\\
 &\simeq& \underset{k \in \Delta^{op}} \lim \Pi^{\et}_\i\left(\left(\boldsymbol{\mu}_n^N\right)^k \times X\right)\left(V\right)\\
 &\simeq& \underset{k \in \Delta^{op}} \lim \Pi^{\et}_\i\left(\left(\boldsymbol{\mu}_n^N\right)^k\right)\left(V\right)\\
 &\simeq& \underset{k \in \Delta_{\le m}^{op}} \lim \Pi^{\et}_\i\left(\left(\boldsymbol{\mu}_n^N\right)^k\right)\left(V\right),
\end{eqnarray*}
where the last equivalence follows from \cite[Lemma 2.21]{KNIRS}.
So, by Theorem \ref{thm:prothesame} we have
\begin{eqnarray*}
\Pi^{\et}_\i\left(\radice{\i}{X}\right)\left(V\right) &\simeq& \underset{n} \colim \Pi^{\et}_\i\left(\radice{n}{X}\right)\left(V\right)\\
&\simeq& \underset{n} \colim \underset{k \in \Delta_{\le m}^{op}} \lim \Pi^{\et}_\i\left(\left(\boldsymbol{\mu}_n^N\right)^k\right)\left(V\right)\\
&\simeq& \underset{k \in \Delta_{\le m}^{op}} \lim \underset{n} \colim \Pi^{\et}_\i\left(\left(\boldsymbol{\mu}_n^N\right)^k\right)\left(V\right)\\
&\simeq& \underset{k \in \Delta_{\le m}^{op}} \lim \Pi^{\et}_\i\left(\left(\boldsymbol{\mu}_\i^N\right)^k\right)\left(V\right)\\
&\simeq& \underset{k \in \Delta^{op}} \lim \Pi^{\et}_\i\left(\left(\boldsymbol{\mu}_\i^N\right)^k\right)\left(V\right)\\
&\simeq&  \Pi^{\et}_\i\left(\underset{k \in \Delta^{op}} \colim\left(\boldsymbol{\mu}_\i^N\right)^k\right)\left(V\right)\\
&\simeq& \Pi^{\et}_\i\left(B\boldsymbol{\mu}_\i\right)\left(V\right),
\end{eqnarray*}
where the $4^{th}$ equivalence follows from Corollary \ref{cor:etaleaffine}. So, on one hand we have that
$$\Pi^{\et}_\i\left(\radice{\i}{X}\right)^{\wedge}_\ell \simeq \Pi^{\et}_\i\left(B\boldsymbol{\mu}_\i\right)^{\wedge}_{\ell},$$ and in the characteristic zero case
$$\widehat{\Pi}^{\et}_\i\left(\radice{\i}{X}\right) \simeq \widehat{\Pi}^{\et}_\i\left(B\boldsymbol{\mu}_\i\right).$$
On the other hand, letting $U:=X^{\triv},$ we have
$$U\cong \mathbb{G}_m^N.$$ So by Corollary \ref{cor:Berner}, we then have
\begin{eqnarray*}
\Pi^{\et}_\i\left(\mathbb{G}_m^N\right)^{\wedge}_{\ell} &\simeq& \Pi^{\et}_\i\left(U\right)^{\wedge}_{\ell}\\
&\simeq& \Shape\left(X_{k\et}\right)^{\wedge}_{\ell}\\
&\simeq& \Pi^{\et}_\i\left(\radice{\i}{X}\right)^{\wedge}_{\ell}\\
&\simeq& \Pi^{\et}_\i\left(B\boldsymbol{\mu}_\i^N\right)^{\wedge}_{\ell},
\end{eqnarray*}
and similarly in the characteristic zero case, but up to profinite completion instead of $\ell$-profinite completion.
\end{proof}

\section{Appendix on Profinite Spaces}
In this appendix, we work out some technical results about profinite spaces which enable us to make local-to-global arguments with them  using hypercovers.

We start by recalling the following result of Lurie:

\begin{theorem}\cite[Theorem E.2.4.1]{SAG}
The composite $$\Profs \hookrightarrow \Pro\left(\Spc\right) \stackrel{\Spc/\left(\blank\right)}{\longlongrightarrow} \Topi$$ is fully faithful and right adjoint to the profinite shape functor.
\end{theorem}

\begin{definition}
A morphism $f:\cX \to \cY$ of profinite spaces is called \textbf{\'etale} if the induced geometric morphism $$\Spc/\cX \to \Spc/\cY$$ is an \'etale geometric morphism. (See \cite[Section 6.3.5]{htt}.)
\end{definition}

\begin{lemma}\label{lem:2coprodet}
Let $\cX$ and $\cY$ be profinite spaces. Then the coprojection $\cX \to \cX \coprod \cY$ is \'etale and $\left(-1\right)$-truncated.
\end{lemma}

\begin{proof}
Let $\cX=\underset{\alpha} \lim j\left(X_\alpha\right)$ and $\cY=\underset{\beta} \lim j\left(X_\beta\right)$ be profinite spaces. Then we have $$\cX \coprod \cY \simeq  \underset{\alpha,\beta} \lim \left(j\left(X_\alpha \coprod Y_\beta\right)\right).$$ As such, we have an identification $$\Spc/\left(\cX \coprod \cY\right) \simeq \underset{\alpha,\beta} \lim \Spc/\left(X_\alpha \coprod Y_\beta\right)$$ in $\Topi.$ However, cofiltered limits in $\Topi$ are computed at the level of underlying $\i$-categories by \cite[Theorem 6.3.3.1]{htt}. Define an object $$\overline{\cX} \in  \underset{\alpha,\beta} \lim \Spc/\left(X_\alpha \coprod Y_\beta\right)$$ by $$\overline{\cX}:=\left(X_\alpha \to X_\alpha \coprod Y_\beta\right)_{\alpha,\beta}.$$ Notice that each morphism $X_\alpha \to X_\alpha \coprod Y_\beta$ in $\Spc$ is a monomorphism, i.e. $\left(-1\right)$-truncated, as it is the inclusion of a union of connected components. It follows that we have a canonical identification
\begin{eqnarray*}
\left(\Spc/\left(\cX \coprod \cY\right)\right)/\overline{\cX} &\simeq& \underset{\alpha,\beta} \lim \Spc/X_\alpha\\
&\simeq& \underset{\beta} \lim \underset{\alpha} \lim \Spc/X_\alpha\\
&\simeq& \underset{\alpha} \lim \Spc/X_\alpha\\
&\simeq& \Spc/\cX.
\end{eqnarray*}
Moreover, it is easy to check that the induced morphism $$\left(\Spc/\left(\cX \coprod \cY\right)\right)/\overline{\cX} \to \left(\Spc/\left(\cX \coprod \cY\right)\right)$$ under the equivalence $\left(\Spc/\left(\cX \coprod \cY\right)\right)/\overline{\cX} \simeq \Spc/\cX$ is the functor $\Spc/\left(\blank\right)$ applied to the coprojection. Hence the coprojection is \'etale. Notice that, by defining $\overline{\cY}$ analogously, we have that $1\simeq \overline{\cX} \coprod \overline{\cY},$ hence the coprojection is also $\left(-1\right)$-truncated.
\end{proof}

\begin{corollary}\label{cor:coprodets}
If $\cX= \underset{\alpha} \coprod \cX_{\alpha}$ is a coproduct of profinite spaces, each coprojection $\cX_{\alpha} \to \cX$ is \'etale and $\left(-1\right)$-truncated.
\end{corollary}

\begin{proof}
Write $\cY:= \underset{\beta \ne \alpha} \coprod \cX_{\beta},$ then $\cX_\alpha \to \cX_\alpha \coprod \cY=\cX$ is \'etale by Lemma \ref{lem:2coprodet}.
\end{proof}

\begin{lemma}\label{lem:pullcop}
Let $\cE$ be an $\i$ topos and suppose that $$1_\cE=\underset{\alpha} \coprod E_\alpha.$$ Suppose furthermore that $f:\cF \to \cE$ is a geometric morphism. Then in $\Topi,$ $$\cF \simeq \underset{\alpha} \coprod \cF/f^*\left(E_\alpha\right).$$
\end{lemma}

\begin{proof}
Recall that
\begin{eqnarray*}
\chi_\cF:\cF &\to& \Topi\\
F &\mapsto& \cF/F
\end{eqnarray*}
preserves colimits  by \cite[Proposition 6.3.5.14]{htt}, and since so does $f^*$ and $f^*$ is left exact, it follows that
\begin{eqnarray*}
\underset{\alpha} \coprod \cF/f^*\left(E_\alpha\right) &\simeq& \cF/\left(\underset{\alpha} \coprod  f^*\left(E_\alpha\right)\right)\\
&\simeq& \cF/\left(f^*\left(\underset{\alpha} \coprod E_\alpha\right)\right)\\
&\simeq& \cF/f^*\left(1_\cE\right)\\
&\simeq& \cF/1_{\cF}\\
&\simeq& \cF.
\end{eqnarray*}
\end{proof}

\begin{corollary}\label{cor:infsp}
Let $\cY \to \underset{\alpha} \coprod \cX_{\alpha}$ be a map in $\Profs.$ Then we have a canonical splitting $$\cY \simeq \underset{\alpha} \coprod \left(\cY \times_{\cX} \cX_{\alpha}\right).$$
\end{corollary}

\begin{proof}
Using the notation of Lemma \ref{lem:pullcop}, let $\cE:=\Spc/\left(\underset{\alpha} \coprod \cX_{\alpha}\right),$ with each $E_\alpha$ corresponding to the \'etale geometric morphism $\Spc/\cX_\alpha \to \Spc/\cX$ via Corollary \ref{cor:coprodets}, and let $\cF:=\Spc/\cY.$

Recall that, by \cite[Remark 6.3.5.8]{htt}, for each $\alpha,$ the following diagram is Cartesian
$$\xymatrix{\cF/f^*\left(E_\alpha\right) \ar[r] \ar[d] & \cF \ar[d]\\
\cE/E_\alpha \ar[r] & \cE.}$$
The result now follows from Lemma \ref{lem:pullcop}, since $\Profs$ is a reflective subcategory of $\Topi.$
\end{proof}

\begin{corollary}
For $\cX,$ $\cY,$ and $\cZ$ arbitrary profinite spaces 
$$\Map\left(\cZ,\cX \coprod \cY\right)\simeq \underset{\cZ= \cZ_\cX \coprod \cZ_{\cY}} \coprod \left(\Map\left(\cZ_\cX,\cX\right) \times \Map\left(\cZ_\cY,\cY\right)\right).$$
\end{corollary}

\begin{proof}
This follows immediately from Corollary \ref{cor:infsp}.
\end{proof}

\begin{lemma}\label{lem:sharkfeet}
Suppose that $\cZ,$ $\cX,$ and $\cY$ are profinite spaces. Suppose further that there is a decomposition $\cZ=\cZ_{\cX} \coprod \cZ_{\cY}$ and maps $$f_\cX:\cZ_{\cX} \to \cX$$ and $$f_\cY:\cZ_{\cY} \to \cY$$ and let $f$ be the induced map $$\cZ \to \cX \coprod \cY.$$ Then
$$\cZ_{\cX} \simeq \cZ\times_{\cX \coprod \cY} \cX.$$
\end{lemma}

\begin{proof}
It suffices to prove that there is a pullback diagram in $\Topi$
$$\xymatrix{\Spc/\cZ_{\cX} \ar[r] \ar[d] & \Spc/\cX \ar[d]\\
\Spc/\cZ \ar[r]^-{f} & \Spc/\left(\cX \coprod \cY\right).}$$
Let $\overline{\cX} \in \Spc/\left(\cX \coprod \cY\right)$ be the object corresponding to the \'etale geometric morphism $$\Spc/\cX \to \Spc/\left(\cX \coprod \cY\right).$$ Then by \cite[Remark 6.3.5.8]{htt}, we have that
$$\left(\Spc/\cZ\right) \times_{\Spc/\left(\cX \coprod \cY\right)} \Spc/\cX \simeq \left(\Spc/\cZ\right)/\left(f^*\left(\overline{\cX}\right)\right).$$ It therefore suffices to prove that $f^*\left(\overline{\cX}\right) \simeq \overline{\cZ_{\cX}}.$

Choose level representations $\left(Z^{\cX}_\alpha \to X_\alpha\right)_\alpha$ and $\left(Z^{\cY}_\beta \to Y_\beta\right)_{\beta}$ of $f_{\cX}$ and $f_{\cY}$ respectively. Then $\left(Z^{\cX}_\alpha \coprod Z^{\cY}_\beta \to X_\alpha \coprod Y_\beta\right)_{\alpha,\beta}$ is a level representation of $f.$ It follows, that in $$\Spc/\cZ \simeq \underset{\alpha,\beta} \lim \Spc/\left(Z^{\cX}_\alpha \coprod Z^{\cY}_\beta\right),$$ we have that 
\begin{eqnarray*}
f^*\left(\overline{\cX}\right)&\simeq& \left(\left(Z^{\cX}_\alpha \coprod Z^{\cY}_\beta\right) \times_{X_\alpha \coprod Y_{\beta}} X_\alpha \to \left(Z^{\cX}_\alpha \coprod Z^{\cY}_\beta\right)\right)_{\alpha,\beta}\\
&\simeq& \left(Z^{\cX}_\alpha \to \left(Z^{\cX}_\alpha \coprod Z^{\cY}_\beta\right)\right)_{\alpha,\beta}\\
&\simeq& \overline{\cZ_{\cX}}.
\end{eqnarray*}
\end{proof}

\begin{definition}
Let $\sC$ be an $\i$-category with fibered products. We say that \textbf{coproducts are universal} in $\sC$ if for all $f:D \to C,$ the functor
\begin{eqnarray*}
f^*:\sC/C &\to& \sC/D\\
\left(E \to C\right) &\mapsto& \left(D \times_{C} E \to D\right)
\end{eqnarray*}
preserves small coproducts.
\end{definition}

\begin{remark}
If $\sC$ is locally Cartesian closed, then coproducts are universal in $\sC,$ since $f^*$ then has a right adjoint, e.g. if $\sC$ is any $\i$-topos.
\end{remark}

\begin{proposition}\label{prop:obs}
Let $f:\cY \to \underset{\alpha} \coprod \cX_{\alpha}\simeq \cX$ be a morphism in $\sC$ and assume that coproduct are universal. Then there is a canonical splitting $\cY \simeq \underset{\alpha} \coprod \left(\cY \times_{\cX} \cX_\alpha\right)$ such that $f$ is induced by the family $$\left(\cY \times_{\cX} \cX_\alpha \to \cX_\alpha\right)_\alpha.$$
\end{proposition}

\begin{proof}
Notice that we have a pullback diagram
$$\xymatrix{ f^*\left(\underset{\alpha} \coprod \cX_{\alpha}\right) \ar[r]^-{\sim} \ar[d] & \cY \ar[d]^-{f}\\
\underset{\alpha} \coprod \cX_{\alpha} \ar[r]^-{\sim} & \cX}$$
Since $f^*$ preserves coproducts, we have that 
\begin{eqnarray*}
\cY &\simeq& f^*\left(\underset{\alpha} \coprod \cX_{\alpha}\right)\\
&\simeq& \underset{\alpha} \coprod f^*\left(\cX_\alpha\right)\\
&\simeq& \underset{\alpha} \coprod \left(\cY \times_{\cX} \cX_\alpha\right).
\end{eqnarray*}
\end{proof}

\begin{remark}
Let $C$ be an object of an $\i$-category $\sC,$ then the canonical functor $$\sC/C \to \Pro\left(\sC\right)/j\left(C\right)$$ induces a cofiltered limit preserving functor $$\Pro\left(\sC/C\right) \to \Pro\left(\sC\right)/j\left(C\right).$$
\end{remark}

Suppose that coproducts are universal in $\sC$ and $\left(\cY^i \to \underset{\alpha} \coprod \cX_\alpha\right)_i$ is a cofiltered system in $\sC/\left(\underset{\alpha} \coprod \cX_\alpha\right).$ By the above remark, there is a canonical map $$f:\underset{i} \lim j\left(\cY^i\right) \to j\left(\underset{\alpha} \coprod \cX_\alpha\right)$$ in $\Pro\left(\sC\right).$ Notice, moreover, that since coproducts are universal, by Proposition \ref{prop:obs}, for each $i$ we have a canonical splitting $$\cY^i\simeq \underset{\alpha} \coprod \cY^i_\alpha$$ together with canonical maps $$f_\alpha:\cY^i_\alpha \to \cX_\alpha$$ inducing $f.$ Suppose furthermore that there is a coproduct preserving functor $$\Pi:\sC \to \Profs$$ which induces a cofiltered limit preserving functor $$\Pi^{\pro}:\Pro\left(\sC\right) \to \Profs.$$

\begin{theorem}
In the situation above, $$\Pi^{\pro}\left(\underset{i} \lim j\left(\cY^i\right)\right) \simeq \underset{\alpha} \coprod \underset{i} \lim \Pi\left(\cY^i_\alpha\right).$$
\end{theorem}

\begin{proof}
Let $\cZ:=\Pi^{\pro}\left(\underset{i} \lim  j\left(\cY^i\right)\right).$ By definition, $$\Pi^{\pro}\left(\underset{i} \lim  j\left(\cY^i\right)\right) \simeq \underset{i} \lim \Pi\left(\underset{\alpha} \coprod \cY^i_\alpha\right),$$ and since $\Pi$ preserves coproducts we have $$\cZ\simeq \underset{i} \lim \underset{\alpha} \coprod \Pi\left(\cY^i_\alpha\right).$$ Similarly, we have a canonical equivalence $$\Pi^{\pro}\left(j\left(\underset{\alpha} \coprod \cX_\alpha\right)\right) \simeq \underset{\alpha} \coprod \Pi\left(\cX_\alpha\right).$$ Since we have a canonical map $$\Pi\left(f\right):\cZ \to \underset{\alpha} \coprod \Pi\left(\cX_\alpha\right),$$ by Corollary \ref{cor:infsp}, we have a canonical splitting 
$$
\cZ \simeq \underset{\alpha'} \coprod \left(\left(\underset{i} \lim \underset{\alpha} \coprod \Pi\left(\cY^i_\alpha\right) \right)  \times_{\underset{\alpha}  \coprod \Pi\left(\cX_\alpha\right)} \Pi \left(\cX_{\alpha'}\right) 
\right).
$$
For each $\alpha'$ let $$\cZ_{\alpha'}:=\left(\left(\underset{i} \lim \underset{\alpha}  
\coprod \Pi\left(\cY^i_\alpha\right)\right)
\times_{\underset{\alpha} \coprod \Pi\left(\cX_\alpha\right)}  \Pi\left(\cX_{\alpha'}\right) 
\right)
.$$ 
Since limits commute with limits, we have
$$\cZ_{\alpha'} \simeq \underset{i} \lim \left(\underset{\alpha} \coprod \Pi\left(\cY^i_\alpha\right)\right)  
\times_{\underset{\alpha} \coprod \Pi\left(\cX_\alpha\right)} 
\Pi\left(\cX_{\alpha'}\right)
$$ 
and by Lemma \ref{lem:sharkfeet}, for each $i$ we have 
$$
\left(\underset{\alpha} \coprod \Pi\left(
\cY^i_\alpha \right) \right)\times_{\underset{\alpha} \coprod \Pi\left(\cX_\alpha\right)} 
\Pi\left(\cX_{\alpha'}\right)
\simeq \Pi\left(\cY^i_{\alpha'}
\right).
$$ Thus $\cZ_{\alpha'}\simeq \underset{i} \lim \Pi\left(\cY^i_{\alpha'}\right),$ and hence 
$$\cZ \simeq \underset{\alpha} \coprod \underset{i} \lim \Pi\left(\cY^i_{\alpha}\right).$$
\end{proof}

The following two corollaries are immediate special cases:

\begin{corollary}\label{cor:fixed2}
Let $X$ be a fine saturated log scheme. Let $\left(U_\alpha \to X\right)_\alpha$ be a collection of \'etale maps. Let $\sC=\Sh\left(\Aff',\et\right)$ and $\Pi=\widehat{\Pi}^{\et}_{\i}.$ Then
$$\widehat{\Pi}^{\et}_{\i}\left(\underset{n} \lim j\left(\underset{\alpha} \coprod \left(\radice{n}{X} \times_{X} U_\alpha\right)\right)\right) \simeq \underset{\alpha} \coprod \underset{n} \lim  \widehat{\Pi}^{\et}_{\i}\left(\radice{n}{X} \times_{X} U_\alpha\right).$$
\end{corollary}

\begin{corollary}
Let $X$ be a fine saturated log scheme locally of finite type over $\mathbb{C}.$ Let $\left(U_\alpha \hookrightarrow X_{top}\right)_\alpha$ be a collection of open subsets of the analytification $X_{top}.$ Using the notation from \cite{KNIRS}, let $\sC=\Hshi\left(\TopC\right),$ and $\Pi=\widehat{\Pi}_{\i}.$ Then
$$\widehat{\Pi}_{\i}\left(\underset{n} \lim j\left(\underset{\alpha} \coprod \left(\radice{n}{X}_{top} \times_{X_{top}} U_\alpha\right)\right)\right) \simeq \underset{\alpha} \coprod \underset{n} \lim  \widehat{\Pi}_{\i}\left(\radice{n}{X}_{top} \times_{X_{top}} U_\alpha\right).$$
\end{corollary}

\begin{remark}
The above corollary fixes a small oversight in the proof of \cite[Theorem 6.4]{KNIRS}. In more detail, the proof there starts by constructing a hypercover $U^\bullet$ of the analytification of the log scheme $X_{an},$ such that for each $n,$ we have
$$U^n=\underset{\alpha} \coprod V_\alpha.$$ In order to justify our reduction of the proof to one such $V_\alpha,$ one needs the above corollary.
\end{remark}




\bibliographystyle{hplain}
\bibliography{profinite}

\end{document}